\documentclass[11pt, reqno]{amsart}
\usepackage{leading}
\usepackage{blindtext}
\usepackage{bm}
\usepackage{lineno}
\newcommand{\todo}[1]{}

\usepackage[new]{old-arrows}

\makeatletter
\DeclareMathAlphabet{\mathpzc}{OT1}{pzc}{m}{it}
\usepackage[margin=2.3cm]{geometry}
\usepackage{amsmath, amssymb, amsfonts, amstext, verbatim, amsthm, mathrsfs}
\usepackage[mathcal]{eucal}
\usepackage{microtype}
\usepackage[colorlinks=true,linkcolor=blue,citecolor=red,urlcolor=blue,citebordercolor={0 0 1},urlbordercolor={0 0 1},linkbordercolor={0 0 1}]{hyperref} 
\usepackage{comment}
\DeclareSymbolFont{bbold}{U}{bbold}{m}{n}
\DeclareSymbolFontAlphabet{\mathbbold}{bbold}

\usepackage[usenames]{color}
\usepackage{aliascnt} 
\usepackage{enumerate}
\usepackage{xspace}
\usepackage{tikz-cd}
\usepackage[nameinlink]{cleveref}

\theoremstyle{plain}

\newcommand{\refnewtheoremn}[4]{
	\newaliascnt{#1}{#2}
	\newtheorem{#1}[#1]{#3}
	\aliascntresetthe{#1}
	\expandafter\providecommand\csname #1autorefname\endcsname{#4}}

\newcommand{\refnewtheorem}[3]{\refnewtheoremn{#1}{#2}{#3}{#3}}

\def\makeCal#1{
	\expandafter\newcommand\csname c#1\endcsname{\mathcal{#1}}}
\def\makeBB#1{
	\expandafter\newcommand\csname b#1\endcsname{\mathbb{#1}}}
\def\makeFrak#1{
	\expandafter\newcommand\csname f#1\endcsname{\mathfrak{#1}}}

\count@=0
\loop
\advance\count@ 1
\edef\y{\@Alph\count@}
\expandafter\makeCal\y
\expandafter\makeBB\y
\expandafter\makeFrak\y
\ifnum\count@<26
\repeat

\newtheorem{theorem}{Theorem}[section]

\refnewtheorem{lemma}{thm}{Lemma}
\refnewtheorem{assumptions}{thm}{Assumptions}
\refnewtheorem{examples}{thm}{Examples}
\refnewtheorem{claim}{thm}{Claim}
\refnewtheorem{cor}{thm}{Corollary}
\refnewtheorem{conj}{thm}{Conjecture}
\refnewtheorem{prop}{thm}{Proposition}
\refnewtheorem{proposition}{thm}{Proposition}
\refnewtheorem{quest}{thm}{Question}
\refnewtheorem{assumption}{thm}{Assumption}

\theoremstyle{definition}
\refnewtheorem{remark}{thm}{Remark}
\refnewtheorem{remarks}{thm}{Remarks}
\refnewtheorem{defn}{thm}{Definition}
\refnewtheorem{definition}{thm}{Definition}
\refnewtheorem{notn}{thm}{Notation}
\refnewtheorem{const}{thm}{Construction}
\refnewtheorem{example}{thm}{Example}
\refnewtheorem{fact}{thm}{Fact}
\refnewtheorem{problem}{thm}{Problem}
\refnewtheorem{convention}{thm}{Convention}

\newcommand{\im}{\operatorname{im}}

\newcommand {\Hom}{\operatorname{Hom}}

\newcommand{\spec}{\operatorname{spec}}

\newcommand{\ul}{Ulrich }
\newcommand{\ses}{short exact sequence }
\newcommand{\les}{long exact sequence }
\newcommand{\wrt}{with respect to }

\numberwithin{equation}{subsection}
\usepackage{expl3}
\usepackage[english]{babel}
\usepackage{xparse}
\ExplSyntaxOn 

\NewDocumentCommand\alotofdots{ D(){\ldotp} O{3}}
{\mathinner{\prg_replicate:nn{#2}{#1}}}
\ExplSyntaxOff

\begin{document}

	\title[ACM, almost Ulrich, and Ulrich bundles arising from Frobenius]{ACM, almost Ulrich, and Ulrich vector bundles arising from Frobenius}
	
	 \author{Jagadish Pine}
	
	\address{Indian Institute of Science Education and Research, Pune,  Main Academic, Dr Homi Bhabha Rd, Pashan, Maharashtra 411008}
	
	\email{math.jagadishpine@gmail.com, jagadish.pine@iiserpune.ac.in}

	\date{}
	
	
	
	\begin{abstract}
		Let $X$ be a smooth projective variety of dimension $n$ over an algebraically closed field of positive characteristic.
		We investigate ACM, almost Ulrich (both ACM and weakly Ulrich), and Ulrich properties of the sheaves of locally exact and closed $i$-forms $B_X^i$, $Z_X^i$. These cohomological properties are closely related to the fundamental concepts of Frobenius splitting and ordinarity.
		Joshi studied the sheaf $B^1_X$ in this context in \cite{MR3920911, MR4206787}, including ACM and almost Ulrich twists of $B^1_X$.
		However, varieties with Ulrich twisted $B_X^1$ have not been classified for $n\ge2$. 
		We classify subcanonical polarized varieties with initialized ACM $B^i_X(aH)$ for $i=1, n$ and prove that, for $n \ge 2$,
		varieties with ACM $B^i_X$ for all $i \le \left\lceil\frac{n}{2}\right\rceil$ are ordinary. 
		 When $X$ is ordinary with trivial $\Omega^1_X$, almost Ulrich twists of $B^i_X$ and $F_*^e(Z_X^i(aH))$ are determined.  
		
		We classify polarized varieties $(X,H)$ with $B_X^i(aH)$ Ulrich when $n=1, 2$, and establish strong necessary criteria for $n \geq 3$, together with an example. We also prove that $Z_X^i(aH)$ is never Ulrich.
		 Finally, we classify polarized abelian varieties with Ulrich $B^1_X \otimes L$ allowing twists by arbitrary line bundles.
		 \end{abstract}

	\maketitle

	\section{Introduction}
	\label{intro}
	
	The study of arithmetically Cohen-Macaulay (ACM) bundles goes back at least to Horrocks’s splitting criterion \cite{MR169877}. An ACM bundle on a polarized variety $(X, H)$ is a vector bundle $E$ all of whose intermediate cohomology groups $H^i(X, E(kH))$ vanish. Ulrich bundles form a distinguished class of ACM bundles. They are globally generated, and attain the upper bound for the number of global sections of initialized ACM bundles. There are generalizations including $\ell$-away ACM, weakly Ulrich, and almost Ulrich bundles that are both weakly Ulrich and ACM. The significance of the existence of weakly Ulrich (resp. Ulrich) bundles is that the Chow form of the underlying variety \wrt the embedding given by $H$ admits a determinantal (resp. linear determinantal) representation \cite{eisenbud2003resultants}.

	Eisenbud and Schreyer in \cite{eisenbud2003resultants} conjectured the existence of an Ulrich bundle on every polarized smooth  projective variety. Their existence has been established for several classes of varieties; see \cite{MR4409698}. In a recent preprint \cite{anghel2026smooth}, a smooth complex polarized surface $(X, H)$ is constructed that admits no Ulrich bundle. The counterexample is generalized in \cite{lopez2026existence, anghel2026polarized}.

In a related direction, whether (twisted) tangent, cotangent, normal, conormal bundles are Ulrich is studied in characteristic zero; see for example \cite{MR4856859}. These questions are in line with semistability and positivity properties of these bundles studied in the past.

	Much of the existing work on ACM and Ulrich bundles has focused on varieties over fields of characteristic $0$. In this article, we study the ACM, almost Ulrich, and Ulrich properties of bundles that naturally arise in positive-characteristic algebraic geometry. These properties are closely related with the fundamental notions of Frobenius splitting and ordinarity of varieties.

	Throughout the article, we work over an algebraically closed field $k$ of characteristic $p>0$. Let $X$ be a smooth projective variety of dimension $n$ over $k$. Let $F:X \longrightarrow X$ be the absolute Frobenius morphism; see Section \ref{second}. For an integer $e \ge 1$, we consider the iterated Frobenius cokernel bundle
	\[
	0 \longrightarrow \mathcal O_X \longrightarrow F_*^e \mathcal O_X \longrightarrow B_e \longrightarrow 0.
	\]
When $e=1$, $B_1$ is the sheaf of locally exact $1$-forms $B^1_X$. Let $d: \Omega^i_X \longrightarrow \Omega^{i+1}_X$ denote the exterior derivative. More generally, we consider the sheaf of locally exact $i$-forms 
	$$B^i_X:=\im\big(F_*d:F_* \Omega^{i-1}_X \longrightarrow F_* \Omega^{i}_X\big)$$
	 and the sheaf of locally closed $i$-forms 
	 $$Z^i_X:=\ker\big(F_*d:F_* \Omega^{i}_X \longrightarrow F_* \Omega^{i+1}_X\big).$$
	 We investigate and classify varieties admitting twists of $B_e$, $B_X^i$, and $Z_X^i$ that are ACM, almost Ulrich, or Ulrich. Let $H$ be a very ample divisor on $X$.  For a vector bundle $E$ on $X$, write $E(kH):=E\otimes \mathcal O_X(kH)$.
	
	Joshi studied the sheaf $B_X^1$ in this context in \cite{MR3920911, MR4206787}. He proved that $B_X^1$ is ACM with respect to $H$ when $X$ is Frobenius split \cite[Theorem 3.1]{MR3920911}, and showed that $B_X^1(H)$ is Ulrich when $X$ is an ordinary curve \cite[Theorem 2.1]{MR3920911}. We discuss further results of Joshi and compare them with ours below.
	
   To the best of our knowledge, the ACM, almost Ulrich and Ulrich properties of twists of $B_X^i$ for $i\geq2$ and of twists of $Z_X^i$ have not previously been studied. Varieties with Ulrich twists of $B_X^1$ have also not been classified when $\dim X \geq 2$.
	
	\begin{theorem} [Theorem \ref{thm3.1}, Proposition \ref{B_eindecomposability}]
		\label{intro:thm1}
		Let $X$ be a smooth projective curve of genus $g$. Let $e \ge 1$ be an integer. Then the following hold:
		\begin{enumerate}
			\item [(1)] The curve $X$ is ordinary if and only if $B_e(H)$ is Ulrich with respect to some (equivalently, every) very ample divisor $H$ on $X$.
			
			\item [(2)] The bundle $B_e$ is semistable. If $g \ge 2$, $B_e$ is simple, and thus, indecomposable.
		\end{enumerate} 
	\end{theorem}
	
	The case (1) generalizes \cite[Theorem 2.1]{MR3920911} which applies to $B_1=B^1_X$. In contrast, for a smooth projective curve $X$,  $Z_X^1(kH)\cong(F_*\omega_X)(kH)$ is not Ulrich for any integer $k$ with respect to $H$; see Theorem \ref{closedformsnotulrich}.

When $X$ is Frobenius split, both $B_X^1$ and $B_X^n$ are ACM with respect to $H$ by Proposition \ref{acmjoshiprop} generalizing \cite[Theorem 3.1]{MR3920911}. However, $B_X^i$ for $ 2 \le i \le n-1$ and $Z_X^i$ need not be ACM even when $X$ is Frobenius split; see Examples \ref{nonacmexample} and \ref{exams3e1}. Our main ACM results are summarized below.

	\begin{theorem} [Proposition \ref{forcedordinarity}, Proposition \ref{frobeniusdirectacm}]
		Let $X$ be a smooth projective variety of dimension $n\geq 2$, and let $H$ be a very ample divisor on $X$. Then the following hold:
		\begin{enumerate}
			\item [(1)] If $B_X^i$ is ACM with respect to $H$ for every $1\leq i\leq \left\lceil\frac{n}{2}\right\rceil$, then $X$ is ordinary.
			
			\item [(2)] Assume $\omega_X \cong \mathcal O_X(\alpha H)$ for $\alpha \in \mathbb Z$. Then $B_X^1$ (equivalently, $B_X^n$) is ACM with respect to $H$ if and only if $X$ is Frobenius split.
			
			\item [(3)] There exists an integer $t_0\geq0$ such that the bundles
			$$ F^e_*\bigl(B_X^i((t_0+1)H)\bigr) \qquad\text{and}\qquad F^e_*\bigl(Z_X^i((t_0+1)H)\bigr) $$
			are ACM with respect to $H$ for every $1\leq i\leq n$ and every integer $e\geq1$ satisfying $p^e>2t_0+1$.
		\end{enumerate}
		\end{theorem}
		
		Part (2) generalizes \cite[Theorem 3.4]{MR3920911} which applies when $\alpha=0$. A classification result for initialized ACM $B^1_X(aH)$ is obtained in Corollary \ref{inicor}. A classification of surfaces with ACM $B^1_X$ is given without any restriction on $\omega_X$ in Corollary \ref{ordiB1}.

   Joshi introduced the notion of an almost Ulrich bundle; see Definition \ref{almostul}. He proved $B_X^1((n-1)H)$ is almost Ulrich when $X$ is Frobenius split and $n=2$ or $3$ \cite[Theorem 5.4]{MR4206787}. When $X$ is ordinary with trivial $\Omega^1_X$, we classify the almost Ulrich twists $B^i_X(sH)$. This shows that a positive power of the Chow form of any ordinary abelian variety with $n \le 3$ is determinantal.

\begin{theorem} [Corollary \ref{acmabelianresult}, Corollary \ref{almostulrichB^i_X}]
		Let $X$ be an ordinary smooth projective variety of dimension $n$ with trivial cotangent bundle, and let $H$ be a very ample divisor on $X$.  For all integers $i,s,a,e$ with $1\leq i\leq n$ and $e\geq1$, the following hold:
		\begin{enumerate}
			\item [(1)]   $B_X^i$ is ACM \wrt $H$. Moreover, $B_X^i(sH)$ is almost Ulrich \wrt $H$ if and only if $n-1\leq s\leq2$. In particular, when $n \ge 4$, no twist $B_X^i(sH)$ is almost Ulrich.
			
			\item [(2)] Assume that $n\ge2$. Then $Z_X^i$ is exactly $(n-1)$-away ACM \wrt $H$.
			Moreover, $F_*^e(Z_X^i(aH))$ is ACM \wrt $H$ if and only if  $p^e \nmid a$.
			
			\item [(3)] $F^e_*\bigl(Z_X^i(aH)\bigr)$ is almost Ulrich \wrt $H$ if and only if  $p^e(n-1)<a<2p^e$. Consequently, when $n \geq 2$, this bundle is almost Ulrich if and only if $n=2$ and $p^e<a<2p^e$.
		\end{enumerate}
	\end{theorem}

	The proof of the theorem involves showing that the de Rham complex $F_*\Omega_X^\bullet$ decomposes as $\bigoplus_\alpha C_\alpha^\bullet$, where $C_\alpha^i\cong L_\alpha\otimes_k\bigwedge^i V$ for a $k$-vector space $V$ and certain line bundles $L_\alpha$. The results then follow by studying the cohomology of the line bundles appearing in this decomposition.

	Now we focus on classification of polarized varieties with Ulrich twists of $B_X^i$ and $Z_X^i$. This is not known even for $B_X^1$ when $\dim X\geq2$. Our main results from Section \ref{four} are summarized below.

	\begin{theorem}
		Let $X$ be a smooth projective variety of dimension $n \ge 1$ over $k$, and let $H$ be a very ample divisor. Let $i, e, a$ be integers with $1 \le i \le n$ and $e \ge 1$. Then the following statements hold:
		\begin{enumerate}
			\item [(1)] Assume that $n=2$. Then $B_{e}(aH)$  is Ulrich with respect to $H$ if and only if 
			\[
			a=1, \qquad e=1, \qquad p=2, \qquad H^2=1,
			\]
			and $(X,H)\cong\bigl(\mathbb P^2,\mathcal O_{\mathbb P^2}(1)\bigr)$.
			
			\item [(2)] Assume that $n \ge 3$. If $B_X^i(aH)$ is Ulrich \wrt $H$, then necessarily
			\[
			n=3,\qquad p=2,\qquad i=2,\qquad a=2.
			\]
			
			\item [(3)]  The bundle $Z_X^i(aH)$ is not Ulrich with respect to $H$.
			\end{enumerate}
		\end{theorem}

		When $n=2$, from part (1), we derive that $B_X^2(aH)$ is Ulrich with respect to $H$ if and only if $a=2$, $p=2$, $H^2=1$, and $(X,H) \cong \bigl(\mathbb P^2,\mathcal O_{\mathbb P^2}(1)\bigr)$ in Corollary \ref{surfaceB2}. Combining $(1)$ and $(2)$, for $n \ge 2$, we get $B^1_X(aH)$ is Ulrich only when $a=1$ and $(X,H)\cong\bigl(\mathbb P^2,\mathcal O_{\mathbb P^2}(1)\bigr)$ in characteristic $2$. Proposition \ref{nonvanishingexactforms} on nonvanishing of sections of $B_X^i(mH)$ is crucially used in the proofs of the above results.

	In Proposition \ref{P3example}, we show that $B_{\mathbb P^3}^2(2)$ is Ulrich with respect to $\mathcal O_{\mathbb P^3}(1)$ for $p=2$, giving an example for part (2). However, we do not know whether $\bigl(\mathbb P^3,\mathcal O_{\mathbb P^3}(1)\bigr)$ is the only polarized threefold in $p=2$ with Ulrich  $B_X^2(2H)$; see Lemma \ref{P^3lemma}.
	
	A more challenging problem is to classify triples $(X, H, L)$ such that $B_X^i \otimes L$ is Ulrich with respect to $H$ for some line bundle $L$ on $X$. Set $M:=L\otimes H^{-1}$. For an elliptic curve $E$, $B_E^1 \otimes L$ is Ulrich with respect to $H$ if and only if $\deg M=0$ and either $M^{\otimes p}\not\cong\mathcal O_E$ or both $M\cong\mathcal O_E$ and $E$ is ordinary; see Remark \ref{sec5:rem1}.

\begin{theorem} [Theorem \ref{sec5:thm1}, Corollary \ref{sec5:cor1}]
	Let  $A$ be an abelian variety of dimension $n\geq2$ over $k$. There exist a very ample line bundle $H$ and a line bundle $L$ on $A$ such that $B_A^1\otimes L$ is Ulrich with respect to $H$ if and only if $A$ is isogenous to a product of $n$ elliptic curves.
\end{theorem}

When $B_A^1\otimes L$ is Ulrich for some $L$ and $H$, the above result gives an isogeny $s:\prod_{j=1}^nK_j\longrightarrow A$.  In Remark \ref{sec5:rem2}, we describe the pullbacks $s^*H$ and $s^*L$ as external tensor products of line bundles on the elliptic curves $K_j$. As a consequence of the above theorem, we prove:
\begin{cor} [Corollary \ref{sec5:cor2}]
	Let $A$ be an abelian variety of dimension $n\geq2$. If the Picard number $\rho(A)<n$, then there do not exist a line bundle $L$ and a very ample divisor $H$ on $A$ such that $B_A^1\otimes L$ is Ulrich with respect to $H$.
	\end{cor}

	\section{Preliminaries}
	\label{second}
	
	Throughout, the base field $k$ is assumed to be algebraically closed  of characteristic $p>0$. Let $X$ be a smooth projective variety over $k$. The absolute Frobenius morphism $F: X \longrightarrow X$ is the identity on the topological space $|X|$ of the variety $X$. On an affine open subset $U=\spec A \subset X$, the ring homomorphism induced by $F$ is given by
	\[
	F^\#:\cO_X(U)\longrightarrow\cO_X(U), \qquad a\longmapsto a^p.
	\]
	Since $\text{char}(k)=p$, this defines a ring homomorphism.

	\subsection{Ordinary and Frobenius split varieties} The morphism $F$ is a finite purely inseparable morphism of degree $p^{\text{dim}(X)}$. Since $X$ is smooth, $F$ is a flat morphism. Consider the following \ses 
	\begin{equation} \label{Fsplitseq}
		0 \longrightarrow \cO_X \longrightarrow F_* \cO_X \longrightarrow B^1_X \longrightarrow 0,
	\end{equation}
	where the cokernel $B^1_X$ is the sheaf of locally exact $1$-forms. Thus, the bundle $B^1_X$ has rank $p^{\dim X}-1$. 
		\begin{definition} (\cite[Definition 2]{mehta1985frobenius})
		A smooth, projective variety $X$ over $k$ is called Frobenius split if the \ses \eqref{Fsplitseq} splits in the category of $\cO_X$-modules.
	\end{definition}
	
		Recall the definition of the sheaf of exact forms:
	\begin{equation}
		B^i_X=\im\big(F_*d:F_* \Omega^{i-1}_X \longrightarrow F_* \Omega^{i}_X\big)
	\end{equation}
	and that of the sheaf of closed forms 
	\begin{equation} \label{closedform}
		Z^i_X=\ker\big(F_*d:F_* \Omega^{i}_X \longrightarrow F_* \Omega^{i+1}_X\big),
	\end{equation}
	where $d$ is the exterior derivative. From these definitions, we get the following \ses of $\cO_X$-modules
	\begin{equation} \label{defseq}
		0 \longrightarrow Z^i_X \longrightarrow F_* \Omega^{i}_X \longrightarrow B^{i+1}_X \longrightarrow 0.
	\end{equation}
	The Cartier operator defines the isomorphism
	\begin{equation} \label{carisom}
		C^i\,:\, Z^i_X/B^i_X \xrightarrow{\sim} \Omega^i_X.
	\end{equation}
	This isomorphism induces the following \ses of $\cO_X$-modules
	\begin{equation} \label{cartierseq}
		0 \longrightarrow B^i_X \longrightarrow Z^i_X \longrightarrow \Omega^i_X \longrightarrow 0.
	\end{equation}
	Let $\dim X\,=\,n$. Then from \eqref{closedform}, we obtain $Z^n_X\,=\,F_* \Omega^{n}_X\,=\,F_* \omega_X$, where $\omega_X$ is the canonical line bundle of $X$. Substituting $i=n$ in \eqref{cartierseq}, we get the following \ses
	\begin{equation} \label{Cartieroperator}
		0 \longrightarrow B^n_X \longrightarrow F_* \omega_X \xrightarrow{C^n} \omega_X \longrightarrow 0.
	\end{equation}
	Substituting $i=0$ in \eqref{defseq}, and using $Z^0_X=\cO_X$, we recover \eqref{Fsplitseq}. Recall the notion of ordinary variety from \cite[Definition 4.12]{illusie1983suites}:
	\begin{definition} [Ordinary variety] 
			\label{ordinarydef}
		A smooth, projective variety $X$ defined over $k$ is called an ordinary variety if $H^i(X, B^j_X)=0$ for all $i \ge 0$ and for all $j>0$.
	\end{definition}

	\subsection{Duality induced by the Cartier operator} The Cartier operator $C^n$ in \eqref{Cartieroperator} defines a perfect $\cO_X$-bilinear pairing for all $ 0 \le j \le n$:
	\begin{equation} \label{laste25}
		F_* \Omega^j_X \otimes F_*\Omega^{n-j}_X \longrightarrow \Omega^n_X=\omega_X, \qquad (w_1,w_2)\longmapsto C^n(w_1 \wedge w_2),
	\end{equation}
	where $C^n:F_* \Omega^n_X=Z^n_X \longrightarrow \Omega^n_X$ is the Cartier morphism. This induces a perfect pairing; see \cite[proof of Lemma~1.1, p.~194]{mehta1987varieties}:
	\[
	B_X^j\otimes B_X^{n-j+1}\longrightarrow\omega_X.
	\]
	Consequently, for $1 \le j \le n$, we obtain
	\begin{equation} \label{cartierduality}
		B_X^j\cong(B_X^{n-j+1})^\vee\otimes\omega_X.
	\end{equation}

	\subsection{ACM and Ulrich bundles}
	We recall some standard notions from the theory of \ul bundles. 
	Let $X$ be a projective variety of dimension $n$. Let $H$ be a very ample divisor on $X$ defining an embedding $X \longhookrightarrow \bP^{N}$.
	\begin{definition} [ACM vector bundles]
		\label{uldefin}
		Let $E$ be a vector bundle on $X$. Then $E$ is called an arithmetically Cohen-Macaulay (ACM) bundle \wrt $H$ if $H^i(X, E(tH))=0$ for all $1 \le i \le (n-1)$, and for all $t \in \bZ$.
	\end{definition} 
	\ul vector bundles have several equivalent characterizations, as follows:
	
	\begin{definition} \cite[Proposition 2.1]{eisenbud2003resultants}\label{Eisenbuddefinition} \cite[Theorem 2.3]{beauville2018introduction}
		\label{ulprop}
		Let $E$ be  a vector bundle on $X$.  Then $E$ is called Ulrich \wrt $H$ if one of the following equivalent conditions is satisfied: 
		\begin{itemize}
			\item [(D1)] The cohomology $H^\bullet(X, E(-jH)) =0$ for $1 \le j \le n$.
			\item [(D2)] If $\pi : X \rightarrow \bP^{n}$ is a finite linear projection, the vector bundle $\pi_* E$ is a trivial bundle on $\bP^{n}$. 
		\end{itemize}
		\end{definition}

	The following facts are standard. We will use them in the subsequent sections. 
	
	Let $X$ be a smooth projective variety of dimension $n \ge 1$. Let $H$ be a very ample divisor on $X$. Let $C$ be a smooth projective complete intersection curve cut out by $(n-1)$ general members of $|H|$. Let $d=H^n$ and $g$ be the genus of the curve $C$. 
	
	\begin{lemma} \label{ulrichdual}
		 Assume that $E$ is an Ulrich bundle on $X$ \wrt $H$. Then the following hold:
		\begin{enumerate}
			\item [(1)] The bundle $E^{\vee} \otimes \omega_X((n+1)H)$ is also Ulrich \wrt $H$.
		\item [(4)] The restriction $E\vert_C$ is Ulrich \wrt $H\vert_{C}$.
			\end{enumerate}
		
	\end{lemma}
	\begin{proof}
		For $1\leq t\leq n$, and for all $q$, Serre duality gives
		\[
		H^q\bigl(X,E^\vee\otimes\omega_X((n+1-t)H)\bigr) \cong H^{n-q}\bigl(X,E(-(n+1-t)H)\bigr)^\vee=0,
		\]
		because $1\leq n+1-t\leq n$. The bundle $E^{\vee} \otimes \omega_X((n+1)H)$ is called the Ulrich dual of $E$.
		
By successive hyperplane restriction, $E|_C$ is Ulrich
		with respect to $H\vert_{C}$; see \cite[Section 3, (3.4), p. 30]{beauville2018introduction},
		whose argument applies in arbitrary characteristic.
		\end{proof}
	
	For further background on \ul bundles, we refer to \cite{MR4409698,beauville2018introduction}.
	
	We will now recall the notion of almost \ul bundles, which are both ACM and weakly Ulrich. Let $X$ be a smooth projective variety of dimension $n$ defined over $k$. Let  $H$ be a very ample divisor on $X$. 
	\begin{definition} \label{almostul} \cite[Section 2, p. 3]{MR4206787}
		A vector bundle $E$ on $X$ is called almost \ul \wrt $H$ if the following conditions are satisfied:
		\begin{itemize}
			\item [(1)] $H^n(X, E(-mH))=0$ for all integers $m \le (n-1)$.
			\item  [(2)] $H^j(X, E(mH))=0$ for $1 \le j \le n-1$ and for all $m \in \bZ$ (i.e., $E$ is ACM).
			\item  [(3)] $H^0(X, E(-mH))=0$ for all integers $m \ge 2$.
		\end{itemize}
	\end{definition}
	
	We now recall the following generalization of ACM bundles, called  $\ell$-away ACM bundles. 
	The study of $\ell$-away ACM bundles is motivated by weakly Ulrich and supernatural bundles, which can have nonvanishing intermediate cohomology.
	
\begin{definition} (\cite[Definition 1.1]{gawron2024ℓ}, \, \cite[Definition 2.9]{bhattacharya2026})
		\label{lawayacm}
		Let $X$ be a smooth projective variety over $k$ and let $H$ be a very ample divisor on $X$ giving the embedding $X \subset  \bP^N$. A vector bundle $E$ on $X$ is called $\ell$-away ACM with respect to $H$ if there are exactly $\ell$ pairs $(i, t) \in \bZ^{2}$ with $0 < i < \text{dim}(X)$ such that $H^i(X, E \otimes \cO_X(H)^{\otimes t}) \ne 0$.
	\end{definition}

	\section{Some auxiliary results}
	We work over an algebraically closed field $k$ of characteristic $p>0$. Let $X$ be a smooth projective variety of dimension $n$ over $k$, and let $H$ be a very ample divisor. We establish a criterion for the nonvanishing of $H^0\big(X,B_X^i(mH)\big)$, and construct a canonical filtration of the sheaves $B_e$. These results will be used in subsequent sections.
	
For a line bundle $L$ on $X$, $F^*L$ admits a canonical connection $\nabla_{\text{can}}$; see \cite[p. 164]{MR2732093}. Let $d_{\nabla_{\mathrm{can}}}$ denote its extension to $F^*L$-valued differential forms. Locally, if $e$ is a frame of $L$, that is, a nowhere vanishing section $e \in H^0(U, L\vert_U)$ for some open set $U \subset X$ giving $\cO_U \simeq L \vert_U$, define
	\begin{equation} \label{e49}
		d_{\nabla_{\mathrm{can}}}:	\Omega^{i-1}_X \otimes F^*(L) \longrightarrow \Omega^{i}_X \otimes F^*(L), \quad d_{\nabla_{\mathrm{can}}}(\omega\otimes F^*e)\,=\, d\omega\otimes F^*e.
	\end{equation}
	This is independent of $e$. Indeed, if $e'=ge$, where $g\in\mathcal O_X^\times(U)$, then $F^*e'=g^pF^*e$, and $d(g^p\omega)=g^p\,d\omega$,
	since $d(g^p)=0$ in characteristic $p$.

	\begin{proposition} \label{nonvanishingexactforms}
		For $1\leq i\leq n$ and for all integer $m\geq0$, if $pm\geq i+1$, then $H^0\bigl(X,B_X^i(mH)\bigr)\neq0$.
	\end{proposition}
	
	\begin{proof}
		Let $	\varphi_H:X\hookrightarrow\mathbb P^N$ be the closed immersion defined by $\mathcal O_X(H)$. We choose a basis $s_0,s_1,\ldots,s_N$
		of $H^0(X,\mathcal O_X(H))$. We consider the affine open subscheme $U=\{s_0\neq0\} \subset X$. Since $\varphi_H$ is a closed immersion, the induced map between cotangent spaces is surjective. Equivalently, the differentials
		\[
		d\left(\frac{s_1}{s_0}\right),\ldots, d\left(\frac{s_N}{s_0}\right)
		\]
		generate $\Omega_X^1|_U$.  Set $x_j=\frac{s_j}{s_0}$. Consequently, at the generic point $\xi$ of $X$, the differentials $dx_1,\ldots,dx_N$ span the $n$-dimensional $k(X)$-vector space $\Omega^1_{X, \xi}$. Since $i\leq n$, after relabelling the $s_j$, we may assume that $0 \ne dx_1\wedge\cdots\wedge dx_i \in \wedge^i \Omega^1_{X, \xi}$ at $\xi$.
		
		On $\mathbb P^N$, with homogeneous coordinates $X_0,\ldots,X_N$, we consider
		\[
		\Theta=\sum_{j=0}^{i-1} (-1)^jX_j\,dX_0\wedge\cdots\wedge \widehat{dX_j}\wedge\cdots\wedge dX_{i-1}.
		\]
		This defines a section $\Theta\in H^0\bigl(\mathbb P^N,\Omega_{\mathbb P^N}^{i-1}(i)\bigr)$; see \cite[Example 2.1, p. 549]{MR4974778}.
		We denote $\theta=\varphi_H^*\Theta$. Then $\theta\in H^0\bigl(X,\Omega_X^{i-1}(iH)\bigr)$.
		
		Since $U=\{s_0\neq0\} \subset X$, we have $U=\varphi_H^{-1}(\{X_0 \ne 0\})$.  Using the following identity defined on $U$
		\[
		dx_j=d\left(\frac{s_j}{s_0}\right)=\frac{s_0\,ds_j-s_j\,ds_0}{s_0^2},
		\]
		it follows that 
		\begin{equation} \label{e48}
			\frac{\theta}{s_0^i}=dx_1\wedge\cdots\wedge dx_{i-1}.
		\end{equation}
		
		On $U$, consider the $(i-1)$-form $\alpha=x_i\,dx_1\wedge\cdots\wedge dx_{i-1}$. Using $x_i=\frac{s_i}{s_0}$ and \eqref{e48}, we obtain
		\[
		s_0^{pm}\alpha=s_0^{pm}\frac{s_i}{s_0}\frac{\theta}{s_0^i}=s_i s_0^{pm-i-1}\theta.
		\]
		Since $pm\geq i+1$, $pm-i-1 \ge 0$. Hence $\eta:=s_i s_0^{pm-i-1}\theta$ is a global section of $\Omega_X^{i-1}(pm H)$.
		
		Since  $\cO_X(pmH) \cong F^* \cO_X(mH)$, $\cO_X(pmH)$ admits a canonical connection $\nabla_{\text{can}}$. Let $d_{\nabla_{\mathrm{can}}}$ denote its extension to $\cO_X(pmH)$-valued differential forms; see \eqref{e49}. We now prove that $d_{\nabla_{\mathrm{can}}}(\eta)$ gives a nonzero section of $\Omega_X^{i}(pm H)$ by showing $d_{\nabla_{\mathrm{can}}}(\eta) \ne 0$ on $U$.
		
		On $U$, the section $s_0^m$ is a frame of $\cO_X(mH)$, and $F^*(s_0^m)=s_0^{pm}$. Therefore,
		\begin{equation}
			\frac{d_{\nabla_{\mathrm{can}}}(\eta)}{s_0^{pm}}=d\left(\frac{\eta}{s_0^{pm}}\right).
		\end{equation}
		Since $\frac{\eta}{s_0^{pm}}\,=\,x_idx_1\wedge\cdots\wedge dx_{i-1}$, we obtain
		\begin{equation}
			\frac{d_{\nabla_{\mathrm{can}}}(\eta)}{s_0^{pm}} =d\left(x_i\,dx_1\wedge\cdots\wedge dx_{i-1}\right)=dx_i\wedge dx_1\wedge\cdots\wedge dx_{i-1}=(-1)^{i-1}dx_1\wedge\cdots\wedge dx_i.
		\end{equation}
		By construction, $dx_1\wedge\cdots\wedge dx_i$ is nonzero at the generic point $\xi$. Hence $d_{\nabla_{\mathrm{can}}}(\eta) \neq0$.
		
		Tensoring the $\cO_X$-linear morphism $F_*d:F_*\Omega_X^{i-1}\longrightarrow F_*\Omega_X^i$ with $\cO_X(mH)$, we get
		\begin{equation} \label{ee41}
			F_*\Omega_X^{i-1}\otimes \cO_X(mH)\longrightarrow F_*\Omega_X^i\otimes \cO_X(mH).
		\end{equation}
		By the projection formula, $F_*\bigl(\Omega_X^j\otimes F^* \cO_X(mH)\bigr) \cong F_*{\Omega_X^j} \otimes \cO_X(mH)$.
		Under these identifications, the map in \eqref{ee41} is the Frobenius pushforward of the map induced by $d_{\nabla_{\mathrm{can}}}$ on $\Omega_X^\bullet\otimes F^*\cO_X(mH)$ in \eqref{e49}.
		
		Since tensoring with a line bundle is exact, the image of \eqref{ee41} is
		\[
		\operatorname{im}(F_*d)\otimes \cO_X(mH)=B_X^i\otimes \cO_X(mH)=B_X^i(mH).
		\]
		Consequently, the nonzero section $d_{\nabla_{\mathrm{can}}}(\eta)$ determines a nonzero global section of $B_X^i(mH)$. Therefore, $H^0\bigl(X,B_X^i(mH)\bigr)\neq0$.
	\end{proof}

	For $e \ge 1$, we consider the sheaves:
	\begin{equation} 
		B_{e}:=\operatorname{coker} \bigl(\mathcal O_X\longrightarrow F_*^e\mathcal O_X\bigr).
	\end{equation}
	When $e=1$, $B_{1}$ is the sheaf of locally exact  $1$-forms $B^1_X$.

	\begin{proposition} \label{filtration of iterated 1 forms}
		Let $X$ be a smooth projective variety. Then there are natural exact sequences
		\begin{equation} \label{B_eseq}
			0\longrightarrow B_{e-1}\longrightarrow B_e\longrightarrow F_*^{e-1}B_1\longrightarrow0.
		\end{equation}
		Consequently, $B_e$ has a filtration with successive quotients $B_1, F_*B_1,\cdots, F_*^{e-1}B_1$.
	\end{proposition}

	\begin{proof}
		We consider the \ses
		\[
		0\longrightarrow \mathcal O_X \xrightarrow{\iota_1} F_*\mathcal O_X \longrightarrow B_1 \longrightarrow 0.
		\]
		Since $F$ is finite, the functor $F_*$ is exact. Applying $F_*^{e-1}$ to this sequence gives
		\begin{equation} \label{e33}
			0\longrightarrow F_*^{e-1}\mathcal O_X \longrightarrow F_*^e\mathcal O_X \longrightarrow F_*^{e-1}B_1 \longrightarrow 0.
		\end{equation}
		Moreover, the map $\iota_e$ factors as $\mathcal O_X\xrightarrow{\iota_{e-1}} F_*^{e-1}\mathcal O_X \longrightarrow F_*^e\mathcal O_X$.
		Consequently, for every $e\geq2$, there is a commutative diagram
		\footnotesize
		\[
		\begin{array}{ccccccccc}
			0 & \longrightarrow & \mathcal O_X
			& \xrightarrow{\iota_{e-1}}
			& F_*^{e-1}\mathcal O_X
			& \longrightarrow & B_{e-1}
			& \longrightarrow & 0 \\[3pt]
			
			& & \downarrow
			& & \downarrow
			& & \downarrow
			& & \\[3pt]
			
			0 & \longrightarrow & \mathcal O_X
			& \xrightarrow{\iota_e}
			& F_*^e\mathcal O_X
			& \longrightarrow & B_e
			& \longrightarrow & 0
		\end{array}
		\]
		\normalsize
		The cokernel of the middle vertical map is $F_*^{e-1}B_1$ by \eqref{e33}. Hence the snake lemma gives the natural exact sequence
		\begin{equation} \label{e34}
			0\longrightarrow B_{e-1} \longrightarrow B_e \longrightarrow F_*^{e-1}B_1 \longrightarrow 0.
		\end{equation}
		Iterating the natural inclusions in \eqref{e34}, we obtain $0\subset B_1\subset B_2\subset\cdots\subset B_e$, and
		\[
		B_j/B_{j-1}\cong F_*^{j-1}B_1
		\]
		for $1\leq j\leq e$, where $B_0=0$. Thus the successive quotients are $B_1, F_*B_1, \cdots, F_*^{e-1}B_1$.
	\end{proof}

	\begin{cor} \label{oneformnonvan}
		For every $e \ge 1$, we have $H^0(X, B_{e}(H)) \ne 0$.
	\end{cor}
	
	\begin{proof}
		Applying Proposition \ref{nonvanishingexactforms} with $m=i=1$, we obtain $H^0(X, B^1_{X}(H)) \ne 0$. Since $B^1_{X}=B_{1} \subseteq B_{e}$ for every $e\ge1$ by Proposition \ref{filtration of iterated 1 forms}, this gives $H^0(X, B_{e}(H)) \ne 0$.		
	\end{proof}
	
	\begin{lemma}
		\label{splitlem}
		Assume $X$ is Frobenius split. Then the sequence \eqref{B_eseq} splits.
	\end{lemma}
	\begin{proof}
	Since $X$ is Frobenius split, there exists a splitting $\sigma:F_*\mathcal O_X\longrightarrow\mathcal O_X$ of \eqref{Fsplitseq}. This
		induces a splitting $F_*^{e-1}\sigma$ of $F_*^{e-1} \iota_1: F_*^{e-1} \cO_X \longrightarrow F_*^{e} \cO_X$, and we obtain the following commutative diagram
		\footnotesize
		\[\begin{tikzcd}
			\cO_X && {F_*^{e} \cO_X} \\
			&& {F_*^{e-1} \cO_X}
			\arrow["{\iota_{e}}", from=1-1, to=1-3]
			\arrow["{\iota_{e-1}}", from=1-1, to=2-3]
			\arrow["{F_*^{e-1}\sigma}", from=1-3, to=2-3]
		\end{tikzcd}\]
		\normalsize
		Thus, after quotienting by the compatible images of $\mathcal O_X$, the map $F_*^{e-1}\sigma$ descends to a splitting
		$B_e\to B_{e-1}$.
	\end{proof}

	\section{Curves with Ulrich twisted $B_e$ and varieties with ACM $B^i_X$, $Z^i_X$}
	\label{three}
Let $X$ be a smooth projective variety of dimension $n$ over an algebraically closed field $k$ of characteristic $p>0$. When $n=1$ and $X$ is ordinary, we show that the vector bundles $B_e(1)$ are Ulrich with respect to any very ample line bundle $\cO_X(1)$. Moreover, we investigate the ACM property of $B^i_X$ and $Z^i_X$ in arbitrary dimensions. Under an additional assumption, we classify varieties for which $B^1_X$ (equivalently, $B^n_X$) is initialized ACM.

 The absolute Frobenius $F:X \longrightarrow X$ is a finite flat morphism of degree $p^n$. Thus, $B^1_X$ is a locally free sheaf of rank $p^n-1$. By \cite[Theorem 2.1, p. 289]{kitadais2008canonical}, we have $c_1(F_*\cO_X)=\frac{p^{n-1}(p-1)}{2}K_X$ in $N^1_{\bQ}(X)$.
Therefore, from \eqref{Fsplitseq}, we obtain
\begin{equation} \label{degreeB^1}
	c_1(B_X^1)=\frac{p^{n-1}(p-1)}{2}K_X,
\end{equation}
in $N^1_{\bQ}(X)$. For a curve, this specializes to $\deg(B_X^1)=(p-1)(g-1)$. For $e \ge 1$, recall the bundles 
\begin{equation}
	B_e=\text{coker}\left(\iota_e:\mathcal O_X\longrightarrow F_*^e\mathcal O_X\right).
\end{equation}
Note that $B_1=B^1_X$ by \eqref{Fsplitseq}. Recall the \ses constructed in Proposition \ref{filtration of iterated 1 forms}
\begin{equation} \label{B_eseq3}
	0\longrightarrow B_{e-1}\longrightarrow B_e\longrightarrow F_*^{e-1}B_1\longrightarrow 0.
\end{equation}

\subsection{\boldmath The sheaves  $B_e$ on curves}  
\label{curclassi}
Let $C$ be a smooth projective curve over $k$ of genus $g$. When $C$ is ordinary, $B_1(1)$ is Ulrich with respect to any very ample line bundle $\cO_C(1)$ by \cite[Theorem 2.1]{MR3920911}. We generalize to $B_e$ for every integer $e\geq 1$, and construct a sequence of indecomposable Ulrich vector bundles on ordinary curves with $g \ge 2$.

\begin{theorem} \label{thm3.1}
	Let $C$ be a smooth projective curve. Then the curve $C$ is ordinary if and only if $B_e(1)$ is Ulrich with respect to some, or equivalently every, very ample line bundle $\mathcal O_C(1)$.
\end{theorem}

\begin{proof} 
	Since $B_1$ is locally free and $F$ is finite and flat, we have  $F_*^{j}B_1$ is locally free. Thus, inductively every $B_e$ is locally free.
	We assume that $C$ is ordinary.	Thus, by Definition \ref{ordinarydef}, $H^i(C, B_1)=0$ for $i=0, 1$. 
	For every $j\geq0$, since $F^j$ is finite, we get
	\begin{equation} \label{e35}
		H^i\bigl(C,F_*^jB_1\bigr)\cong H^i(C,B_1)=0,\qquad i=0,1.
	\end{equation}
	Using the short exact sequences \eqref{B_eseq3}, and using \eqref{e35}, induction on $e$ yields
	\begin{equation} \label{e36}
		H^0(C,B_e)=H^1(C,B_e)=0
	\end{equation}
	for every $e\geq1$. Now let $\mathcal O_C(1)$ be any very ample line bundle. Then we have $H^i(C,B_e(1) \otimes \cO_C(-1))=H^i(C,B_e)=0$ for $i=0, 1$ by \eqref{e36}. Therefore, by condition $(D_1)$ in Definition \ref{ulprop}, $B_e(1)$ is an Ulrich bundle \wrt $\cO_C(1)$.
	
	Conversely, assume that for some $e$, $B_e(1)$ is Ulrich with respect to some very ample line bundle $\cO_C(1)$. Thus, $H^i(C, B_e)=0$ for $i=0,1$. Since $B_1 \subset B_e$, we get $H^0(C, B_1)=0$. Moreover, taking the \les in cohomology for \eqref{B_eseq3}, we obtain $H^1(C, F^{e-1}_* B_1)=0$. This gives $H^1(C, B_1)=0$. Thus, $C$ is ordinary.
\end{proof}

For the following two results, we do not assume the curve $C$ to be ordinary.

\begin{lemma}
	Let $C$ be a smooth projective curve of genus $g$. For every $e \ge 1$, the bundle $B_e$ is semistable of slope $(g-1)$. 
\end{lemma}

\begin{proof}
	By Proposition \ref{filtration of iterated 1 forms}, the canonical filtration
	\[
	0=B_0\subset B_1\subset\cdots\subset B_e
	\]
	has successive quotients $	B_i/B_{i-1}\simeq F_*^{i-1}B_1$.  When $C=\bP^1$, $B_e \otimes \cO_{\bP^1}(1)$ is Ulrich \wrt $\cO_{\bP^1}(1)$. Thus, $B_e \cong \cO_{\bP^1}(-1)^{\oplus (p^e-1)}$, which is semistable.
	
	Assume $g \ge 1$. The bundle $B_1$ is semistable by \cite[Theorem 4.1.1, p. 119]{raynaud1982sections}, and $F^i_* B_1$ is also semistable by \cite[Theorem 2.2]{sun2008direct}. Moreover, $\mu(F^i_* B_1)=g-1$
	for every $i$. Thus, $\mu(B_e)=g-1$. By induction, we assume $B_n$ is semistable. We want to show that $B_{n+1}$ is also semistable. Assume by contradiction that $B_{n+1}$ is not semistable, and let $F \subset B_{n+1}$ with $\mu(F) > g-1$ be the maximal destabilizing subbundle. We consider the induced map $F \longrightarrow F^n_* B_1$. Since $\mu(F) > g-1=\mu(F^n_* B_1)$, this map is zero. Thus from \eqref{B_eseq3}, it follows that $F \subset B_n$. This contradicts the semistability of $B_n$. Thus, $B_{n+1}$ is semistable, and consequently $B_e$ is semistable for every integer $e \ge 1$. 
	
	If $e \ge 2$, then $B_1 \subset B_e$ with $\mu(B_1)=\mu(B_e)$. Thus, for $e \ge 2$, $B_e$ is strictly semistable.
\end{proof}

\begin{proposition} \label{B_eindecomposability}
	Let $C$ be a smooth projective curve of genus $g \ge 2$. For every $e \ge 1$, the bundle $B_e$ is simple. Consequently, $B_e$ is indecomposable. 
\end{proposition}

\begin{proof} 
Denote $E:=F_*^e \cO_C$. The defining \ses for $B_e$ is
\[
0\longrightarrow\mathcal O_C \xrightarrow{u} E \xrightarrow{q} B_e \longrightarrow0.
\]
Since $g \ge2$, $E$ is a stable bundle by \cite[Theorem 2.2]{sun2008direct}. Thus $\operatorname{End}_C(E)=k$.
Since $E$ is locally free, applying $\mathcal Hom(E,-)$ induces the following \ses
\begin{equation} \label{lre2}
	0\longrightarrow E^\vee \xrightarrow{\,j\,}\mathcal E nd(E) \longrightarrow\mathcal H om(E,B_e) \longrightarrow 0,
	\qquad j(\lambda)=u\circ\lambda.
\end{equation}
We claim that \eqref{lre2} splits. For a local section $a$ of the $\mathcal O_C$-algebra $E$, let $m_a: E\to E$ denote multiplication by $a$. Define
\begin{equation} \label{lre3}
	R:\mathcal E nd(E)\longrightarrow E^\vee, \qquad R(T)(a)=\operatorname{tr}(T\circ m_a).
\end{equation}
For a local section $\beta$ of $\cO_X$, by the $\mathcal O_C$-linearity of trace, we obtain
\[
R(T)(\beta\cdot a)=\beta R(T)(a), \qquad R(\beta T)=\beta R(T).
\]
Thus, the map $R$ is $\cO_C$-linear. For a local section $\lambda$ of $E^\vee$, the endomorphism
$(u\circ\lambda)\circ m_a$ is given by
\begin{equation} \label{lre4}
	b\longmapsto\lambda(ab)\cdot 1,
\end{equation}
where $1$ is the unit section of $E$ and the dot denotes
the $\mathcal O_C$-module action on $E$.
Since the morphism \eqref{lre4} is of rank-one, its trace is $(\lambda\circ m_a)(1)=\lambda(a)$. Therefore
\[
R(j(\lambda))(a)=\lambda(a),
\]
so $R\circ j=\operatorname{id}_{E^\vee}$, proving the claim.
Thus, taking global sections of the split exact sequence \eqref{lre2}, we obtain a surjection
\[
\operatorname{End}_C(E) \longrightarrow\operatorname{Hom}_C(E,B_e), \qquad T\longmapsto q\circ T.
\]
Since $\operatorname{End}_C(E)=k$ and $\operatorname{id}_E$ maps to $q\ne0$, it follows that $\operatorname{Hom}_C(E,B_e)=kq$.

Thus, for any $\varphi\in\operatorname{End}_C(B_e)$, there exists $c\in k$ such that $\varphi\circ q=cq$. As $q$ is surjective, $\varphi=c \cdot \operatorname{id}_{B_e}$. Hence $\operatorname{End}_C(B_e)=k$, so $B_e$ is simple. Consequently, $B_e$ is indecomposable.
\end{proof}

\begin{remark}[On the splitting of the sequence \eqref{B_eseq3} on curves]
		By Lemma \ref{splitlem}, the sequence \eqref{B_eseq3} splits when $X=\bP^1$ or $X$ is an ordinary elliptic curve. 
		By Proposition \ref{B_eindecomposability}, $B_e$ is indecomposable if $g \ge 2$. Thus, the sequence \eqref{B_eseq3} is nonsplit whenever $g \ge 2$.
	\end{remark}

\begin{remark} [Ulrich but not strongly semistable]
	Let $C$ be a smooth projective curve of genus $g\ge 2$ over an algebraically closed field in characteristic $p>2$. Let $H$ be a very ample divisor on $C$. Then it is known that $F^*B_C^1$ is not semistable \cite[Remark 4.1.2(2), p. 119]{raynaud1982sections}. 
Consequently, $F^* (B^1_C(H)) \cong (F^*B_C^1)(pH)$ is also not semistable. Therefore, when $C$ is ordinary, $B^1_C(H)$ is Ulrich \wrt $H$ by Theorem \ref{thm3.1} but not strongly semistable.
\end{remark}

\subsection{Sheaves of exact and closed forms on smooth projective varieties}
\label{threetwo}
Let $X$ be a smooth projective variety of dimension $n \ge 1$. We show that $B^i_X$ being ACM forces $X$ to be ordinary, and obtain a classification result for varieties with initialized ACM $B^1_X$. For $n \ge 3$ and $i \ge 2$,  $B^i_X$ and $Z^i_X$ need not be ACM even when $X$ is ordinary or Frobenius split. Finally, sufficiently high Frobenius direct images of twisted $B^i_X$ and $Z^i_X$ are shown to be ACM.

By substituting $j=1$ in \eqref{cartierduality}, we obtain
\[
B_X^1\cong(B_X^n)^\vee\otimes\omega_X.
\]
Therefore, by Serre duality,
\begin{equation} \label{lre1}
H^j(X,B_X^n\otimes L^t) \cong H^j\left(X,(B_X^1)^\vee\otimes\omega_X\otimes L^t\right) \cong H^{n-j}(X,B_X^1\otimes L^{-t})^\vee
\end{equation}
for $1\leq j\leq n-1$ and for any line bundle $L$ on $X$. As a corollary, we get the following.
\begin{cor} \label{1nequi}
	Let $X$ be a smooth projective variety, and let $H$ be a very ample divisor. Then $B^1_X$ is ACM \wrt $H$ if and only if $B^n_X$ is ACM \wrt $H$.
\end{cor}

Combining the preceding corollary with \cite[Theorem 3.1]{MR3920911}, we obtain the following. 

\begin{proposition} \label{acmjoshiprop}
	Let $X$ be a smooth projective and  Frobenius split variety of dimension $n$. Then $B^1_X$ and $B^n_X$ are ACM \wrt any very ample line bundle. 
\end{proposition}

Let $H$ be a very ample divisor on $X$. Assume that the pair $(X,H)$ is subcanonical; that is, $\omega_X\cong\cO_X(\alpha H)$ for some integer $\alpha$. The following result generalizes \cite[Theorem 3.4]{MR3920911}, which applies when $\alpha=0$.

\begin{theorem} \label{subcan}
	Assume $n \ge 2$. Let $H$ be a very ample divisor on $X$ such that $\omega_X \cong \cO_X(\alpha H)$ for some integer $\alpha$. Then $B_X^1$ (equivalently, $B_X^n$) is ACM with respect to $H$ if and only if $X$ is Frobenius split.
\end{theorem}

\begin{proof}
	If $B_X^1$ is ACM \wrt $H$, then
	\[
	H^{n-1}(X,B_X^1\otimes\omega_X)=H^{n-1}(X,B_X^1(aH))=0.
	\]
	Since $B_X^1$ is locally free, Serre duality gives
	\[
	\operatorname{Ext}^1_{\mathcal O_X}(B_X^1,\mathcal O_X) \simeq H^1(X,(B_X^1)^\vee) \simeq H^{n-1}(X,B_X^1\otimes\omega_X)^\vee=0.
	\]
	Therefore the sequence \eqref{Fsplitseq} splits, and thus, $X$ is Frobenius split. If $B^n_X$ is is ACM \wrt $H$, then Corollary \ref{1nequi} gives $X$ is Frobenius split. The converse follows from Proposition \ref{acmjoshiprop}.
\end{proof}

If $X$ is Frobenius split, then $H^i(X,B_X^1)=0$ for every integer $i\ge0$ by \eqref{Fsplitseq}, and hence $H^0(X,B_X^n)=0$ by \eqref{lre1}. Moreover, applying Proposition \ref{nonvanishingexactforms} with $i=1$ gives $H^0(X,B_X^1(H))\ne0$. Thus, $B_X^1(H)$ is initialized whenever $X$ is Frobenius split. Together with Theorem \ref{subcan}, this yields the following result:

\begin{cor} \label{inicor}
	Let $H$ be a very ample divisor on $X$ such that $\omega_X \cong \cO_X(\alpha H)$ for some integer $\alpha$. Let $a \in \bZ$. Then $B_X^1(aH)$ is initialized ACM with respect to $H$ if and only if $a=1$ and $X$ is Frobenius split.
\end{cor}
By Theorem \ref{subcan}, $B_X^n$ is ACM when $X$ is Frobenius split. Thus, there exists a unique integer $a_0$ such that $B_X^n(a_0H)$ is initialized ACM. This yields a result analogous to Corollary \ref{inicor}, characterizing when $B_X^n(aH)$ is initialized ACM.

Thus, the above results provide numerous examples. For $n\ge2$, a general smooth hypersurface $X\subset\mathbb P^{n+1}$ of degree $d\le n+1$ is known to be Frobenius split, and hence the bundle $B_X^1$ is ACM with respect to the hyperplane divisor.

The following result shows that the ACM property of $B^i_X$ forces ordinarity without any restriction on $\omega_X$.

\begin{proposition} \label{forcedordinarity}
	Let $X$ be a smooth projective variety of dimension $n\geq 2$. If $B_X^i$ is ACM with respect to a very ample divisor $H$ for every
	$1\leq i\leq \left\lceil\frac{n}{2}\right\rceil$, then $X$ is ordinary.
\end{proposition}

\begin{proof}
	For $1\leq i\leq n$, from \eqref{defseq} and \eqref{cartierseq}, we have the following exact sequences
	\begin{equation} \label{e341}
		0\longrightarrow B_X^i\longrightarrow Z_X^i \longrightarrow \Omega_X^i\longrightarrow 0
	\end{equation}
	\begin{equation} \label{e342}
		0\longrightarrow Z_X^i\longrightarrow F_*\Omega_X^i \longrightarrow B_X^{i+1}\longrightarrow 0,
	\end{equation}
	where $B_X^{n+1}=0$.  
	We have $B_X^j\cong(B_X^{n-j+1})^\vee\otimes\omega_X$ by \eqref{cartierduality}. Therefore, by Serre duality, we obtain
	\begin{equation} \label{acmupperhalf}
		H^q(X, B_X^j(kH)) \cong H^q(X, (B_X^{n-j+1})^\vee\otimes\omega_X (kH)) \cong H^{n-q}(X, (B_X^{n-j+1})(-kH))
	\end{equation}
	Since $B_X^i$ is ACM for every
	$1\leq i\leq \left\lceil\frac{n}{2}\right\rceil$, by \eqref{acmupperhalf}, it follows that  $B_X^i$ is ACM for all $1 \le i \le n$. Since $n\geq2$, this gives
	\begin{equation} \label{e343}
		H^q(X,B_X^i)=0, \qquad 1\leq q\leq n-1.
	\end{equation}
	Therefore, taking the \les in cohomology for \eqref{e341}, we obtain
	\[
	0\longrightarrow H^0(X,B_X^i) \longrightarrow H^0(X,Z_X^i) \longrightarrow H^0(X,\Omega_X^i) \longrightarrow 0.
	\]
	On the other hand, \eqref{e342} gives an injection $H^0(X,Z_X^i)\longrightarrow H^0(X,F_*\Omega_X^i)$.
	Since $F$ is a finite map, $h^0(X,F_*\Omega_X^i)=h^0(X,\Omega_X^i)$. This gives
	\[
	h^0(X,\Omega_X^i)+h^0(X,B_X^i)=h^0(X,Z_X^i)\leq h^0(X,\Omega_X^i).
	\]
	Hence $H^0(X,B_X^i)=0$ for every $i$. Thus, substituting $q=n$, $j=i$, and $k=0$ in \eqref{acmupperhalf} gives
	\[
	H^n(X,B_X^i)\cong H^0(X,B_X^{n-i+1})=0.
	\]
	
	Thus $H^q(X,B_X^i)=0$ for every $q\geq0$ and every $1\leq i\leq n$. Therefore, $X$ is ordinary.
\end{proof}

\begin{cor} \label{ordiB1}
Assume $n=2$. Then $B^1_X$ is ACM for some very ample divisor if and only if $X$ is ordinary.
\end{cor}
\begin{proof}
	If $X$ is ordinary, we have $H^i(X, B^1_X)=0$ for all $i$. Let $H$ be a very ample divisor. Then by Serre duality, we have 
	\[
	H^1(X,B^1_X(rH))=0 \qquad (|r| \gg 0).
	\]
	Thus, for sufficiently large integer $m \gg 0$, $B^1_X$ is ACM \wrt $mH$. The converse follows from Proposition \ref{forcedordinarity}.
\end{proof}

By Proposition \ref{acmjoshiprop}, if $X$ is Frobenius split, then both $B^1_X$ and $B^n_X$ are ACM \wrt any very ample divisor. However, $B^i_X$ for $2 \le i \le n-1$ may not be ACM. The following gives an example.

\begin{example} [Frobenius split Calabi-Yau threefold with non-ACM $B^2_X$] 
	\label{nonacmexample}
	Consider the Fermat quintic threefold $X=\left\{x_0^5+x_1^5+x_2^5+x_3^5+x_4^5=0\right\}\subset\mathbb P_k^4$ over $k=\overline{\mathbb F}_{11}$. Let $\mathcal O_X(1)$ be the very ample line bundle $\cO_{\bP^4}(1)\vert_X$. 
	It is smooth, and by adjunction, $\omega_X\cong\mathcal O_X$. Moreover, $X$ is Frobenius split. Indeed, the coefficient of
	$(x_0x_1x_2x_3x_4)^{10}$ in $(x_0^5+\cdots+x_4^5)^{10}$ is $\frac{10!}{(2!)^5}\equiv1\pmod{11}$. Thus, by \cite[Example 1.5, p. 185]{schwede2026frobenius}, $X$ is Frobenius split. 
	Nevertheless, $B_X^2$ is not ACM with respect to $\cO_X(1)$ as we now show.
	
	Let $i:X\hookrightarrow\mathbb P^4$ be the inclusion. We consider the \ses
	\begin{equation} \label{e346}
		0\longrightarrow \mathcal O_{\mathbb P^4}(-5) \longrightarrow \mathcal O_{\mathbb P^4} \longrightarrow i_*\mathcal O_X \longrightarrow 0.
	\end{equation}
	Twisting \eqref{e346} by $\mathcal O_{\mathbb P^4}(m)$ and taking the \les in cohomology, we obtain
	\begin{equation} \label{e350}
		H^q(X,\mathcal O_X(m))=0 \qquad(q>0,\ m>0), \qquad H^3(X,\mathcal O_X)\cong k.
	\end{equation} 
	
	The restricted Euler and conormal sequences are
	\begin{equation} \label{e347}
		0\longrightarrow \Omega_{\mathbb P^4}^1(m)|_X \longrightarrow \mathcal O_X(m-1)^{\oplus5} \longrightarrow \mathcal O_X(m) \longrightarrow0,
	\end{equation}
	\begin{equation} \label{e348}
		0\longrightarrow \mathcal O_X(m-5) \longrightarrow \Omega_{\mathbb P^4}^1(m)|_X \longrightarrow \Omega_X^1(m)
		\longrightarrow0.
	\end{equation}
	For $m\geq2$, taking the \les in cohomology for \eqref{e347}, and using \eqref{e350}, we get
	\begin{equation} \label{ee5027}
		H^q\bigl(X,\Omega_{\mathbb P^4}^1(m)|_X\bigr)=0 \qquad(q=2,3).
	\end{equation}
	The long exact sequence in cohomology associated with \eqref{e348}, together with \eqref{ee5027} gives
	\begin{equation}
		H^2\bigl(X,\Omega_X^1(m)\bigr) \cong H^3\bigl(X,\mathcal O_X(m-5)\bigr) \qquad(m\geq2).
	\end{equation}
	In particular, $H^2\bigl(X,\Omega_X^1(5)\bigr) \cong H^3\bigl(X,\mathcal O_X\bigr) \cong k$. Using the projection formula, we get
	\[
	H^2\bigl(X,(F_*\Omega_X^1)(5)\bigr) \cong H^2(X, F_*(\Omega_X^1 \otimes F^* \cO_X(5))) \cong H^2\bigl(X,\Omega_X^1(55)\bigr)=0.
	\]
	Moreover, twisting \eqref{Fsplitseq}
	by $\mathcal O_X(5)$, and taking the \les in cohomology gives $H^3(X,B_X^1(5))=0$, since
	$H^3(X,\mathcal O_X(55))=0$.
	Thus, twisting \eqref{defseq} and \eqref{cartierseq} by $\mathcal O_X(5)$, and taking cohomology now give surjections
	\[
	H^1\bigl(X,B_X^2(5)\bigr)
	\twoheadrightarrow H^2\bigl(X,Z_X^1(5)\bigr)
	\twoheadrightarrow H^2\bigl(X,\Omega_X^1(5)\bigr)
	\cong k.
	\]
	Hence $B_X^2$ is not ACM with respect to $\mathcal O_X(1)$.
\end{example}

By \eqref{defseq}, we have $Z^n_X =F_* \omega_X$. Thus, when $X=\bP^n$, $Z^n_X$ is ACM \wrt any very ample divisor on $\bP^n$. However, even if $X$ is Frobenius split, $Z^i_X$ need not be ACM in general. For example, when $X=\bP^2$, we have $H^1(\bP^2, Z^1_{\bP^2}) \ne 0$, and thus $Z^1_{\bP^2}$ is not ACM \wrt any very ample divisor; see Example \ref{exams3e1}.

However, the following result shows that every vector bundle admits a twist whose sufficiently high Frobenius direct images are ACM.

\begin{proposition}
	Let $H$ be a very ample divisor on $X$. Let $E$ be a vector bundle on $X$. Then there exist $e \ge 1$ and $a \in \bZ$ such that $F^e_* (E(aH))$ is ACM \wrt $H$. 
\end{proposition}

\begin{proof}
	Recall that $\dim X=n$.	We define
	\[
	S(E)=\left\{t\in\mathbb Z:H^i(X, E(tH))\neq0\text{ for some }1\leq i\leq n-1\right\}.
	\]
	By Serre vanishing and Serre duality, it follows that $S(E)$ is finite. By the projection formula $$F_*^e(E(aH))(tH) \cong F_*^e \big(E(aH) \otimes (F^e)^*\cO_X(tH) \big) \cong F^e_* E((a+p^et)H).$$ 
	This gives 
	\[
	H^i\left(X,F_*^e(E(aH))(tH)\right)=H^i\left(X,E((a+p^et)H)\right).
	\]
	Therefore, if $a\in\mathbb Z$ and $e\geq1$ satisfy
	\[
	(a+p^e\mathbb Z)\cap S(E)=\varnothing,
	\]
	then $F_*^e\bigl(E(aH)\bigr)$ is ACM with respect to $H$. In particular, if $S(E)\subseteq[-t_0,t_0]$ for some integer $t_0 \ge 0$, then one may take $a=t_0+1$ and any $e$ satisfying $p^e>2t_0+1$.
\end{proof}

Applying the proposition to $B_X^i$ and $Z_X^i$, we obtain the following corollary.
\begin{cor} \label{frobeniusdirectacm}
	Let $H$ be a very ample divisor on $X$.	There exists an integer $t_0\geq0$ such that
	$$ F^e_*\bigl(B_X^i((t_0+1)H)\bigr) \qquad\text{and}\qquad F^e_*\bigl(Z_X^i((t_0+1)H)\bigr) $$
	are ACM with respect to $H$ for every $1\leq i\leq n$ and every integer $e\geq1$ satisfying $p^e>2t_0+1$.
\end{cor}
\begin{proof}
	Apply the proof of the proposition to the vector bundle $ E=\bigoplus_{i=1}^{n}\bigl(B_X^i\oplus Z_X^i\bigr)$.
	The same choice of $t_0$ then works for every summand.
\end{proof}

\begin{example} \label{exams3e1}
	Consider $X=\bP^2$, and $H=\cO_{\bP^2}(1)$. Then using the Euler sequence
	$$0 \longrightarrow \Omega^1_{\bP^2} \longrightarrow \cO_{\bP^2}(-1)^{\oplus 3} \longrightarrow \cO_{\bP^2} \longrightarrow 0,$$
	it follows that $H^1(\bP^2, \Omega^1_{\bP^2}(t)) = 0$ if and only  if $t \ne 0$. Since $\bP^2$ is Frobenius split, $B^1_{\bP^2}$ is ACM \wrt $\cO_{\bP^2}(1)$. Using the \ses from \eqref{cartierseq}:
	$$0\longrightarrow B_{\bP^2}^1 \longrightarrow Z_{\bP^2}^1 \longrightarrow\Omega_{\bP^2}^1 \longrightarrow0,$$
	we get  $S(Z^1_{\bP^2})=\{0\}$, since $H^2(\bP^2, B^1_{\bP^2})=0$.  Thus, $F^e_*(Z^1_{\bP^2}(a))$ is ACM \wrt $\cO_{\bP^2}(1)$ if and only if $p^e \nmid a$.
\end{example}

	\section{The sheaves $B_X^i$ and $Z_X^i$ on ordinary varieties with trivial cotangent bundle}
	\label{sec3.3}
	
	Let $X$ be a smooth projective variety of dimension $n$. If $X$ is Frobenius split and $n=2$ or $3$, then a suitable twist of $B_X^1$ is almost Ulrich; see \cite[Theorem 5.4]{MR4206787}. In this section, we investigate the ACM and almost Ulrich properties of $B^i_X$ and $Z^i_X$ under the assumption that $X$ is ordinary and $\Omega_X^1 \cong \mathcal O_X^{\oplus n}$.

Let $\Gamma$ be the group of all $p$-torsion line bundles
\[
\Gamma:=\text{Pic}(X)[p]=\left\{ [L]\in\text{Pic}(X)\mid L^{\otimes p}\cong\mathcal O_X\right\}.
\]
Since $X$ is ordinary with $\Omega_X^1\cong\mathcal O_X^{\oplus n}$, by \cite[Lemma 1.1, p. 193]{mehta1987varieties}, $X$ is Frobenius split. Thus, \cite[Proposition 3.4]{ejiri2019characterization} gives
\[
H^0(X,\Omega_X^1) \cong \text{Pic}(X)[p]\otimes_{\mathbb Z}k.
\]
Since $H^0(X,\Omega_X^1)\cong k^n$, it follows that $\Gamma=\text{Pic}(X)[p]$ is an $n$-dimensional $\mathbb F_p$-vector space. Hence
\[
\Gamma\cong(\mathbb Z/p\mathbb Z)^n \quad\text{and}\quad \#\Gamma=p^n.
\]
We choose representatives $L_\alpha$ for $\alpha\in\Gamma$, with $L_0=\mathcal O_X$. Since $X$ is Frobenius split and $\#\Gamma=p^n=\deg F$, by \cite[Lemma 3.1]{ejiri2019characterization} we get
\begin{equation} \label{e310}
	F_*\mathcal O_X \cong \bigoplus_{\alpha\in\Gamma}L_\alpha.
\end{equation}

\begin{proposition} \label{abdecom}
	Let $X$ be a smooth projective ordinary variety of dimension $n$ with $\Omega_X^1\cong\mathcal O_X^{\oplus n}$.	 Then for $1\leq i\leq n$,
	we have		
	\[
	B_X^i \cong \bigoplus_{\alpha\neq0} L_\alpha^{\oplus\binom{n-1}{i-1}}, \quad Z_X^i \cong \mathcal O_X^{\oplus\binom ni} \oplus \bigoplus_{\alpha\neq0} L_\alpha^{\oplus\binom{n-1}{i-1}}.
	\]
	In particular, $Z_X^i\cong \Omega_X^i\oplus B_X^i$ and $B_X^i\cong
	\left(B_X^1\right)^{\oplus\binom{n-1}{i-1}}$.
\end{proposition}

\begin{proof}
	We choose a $k$-vector space $V$ of dimension $n$ and an isomorphism $\Omega_X^1\cong V\otimes_k\mathcal O_X$.
	Then
	\[
	\Omega_X^i \cong \bigwedge^iV\otimes_k\mathcal O_X \cong \mathcal O_X^{\oplus\binom ni}.
	\]
	Consequently, by \eqref{e310}, we get 
	\begin{equation}
		F_*\Omega_X^i \cong \big(\bigoplus_{\alpha\in\Gamma} L_\alpha \big) \otimes_k\bigwedge^iV.
	\end{equation}
	If $\alpha\neq\beta$, then $L_\alpha^{-1}\otimes L_\beta \not \cong \cO_X$, and 
	\begin{equation} \label{e311}
		\text{Hom}(L_\alpha,L_\beta)=H^0\left(X,L_\alpha^{-1}\otimes L_\beta\right)=0.
	\end{equation}
	Indeed, a nonzero section of $L_\alpha^{-1}\otimes L_\beta$ defines a nonzero effective divisor $D$. Since $L_\alpha^{-1}\otimes L_\beta$ is torsion, $mD=(f)$ is principal for some $m>0$ and some rational function $f$. Then $f$ has no poles and is therefore constant, forcing $D=0$, giving a contradiction.
	
	We consider the de Rham complex $F_*\Omega_X^\bullet$, whose differentials $d:F_*\Omega_X^i\longrightarrow F_*\Omega_X^{i+1}$ are $\mathcal O_X$-linear. It follows from \eqref{e311} that $d$ preserves the $L_\alpha$-isotypic summands. Thus the complex decomposes as
	\begin{equation} \label{e312}
		F_*\Omega_X^\bullet \cong \bigoplus_{\alpha\in\Gamma}C_\alpha^\bullet,
	\end{equation}
	where $C_\alpha^i= L_\alpha\otimes_k\bigwedge^iV$. Since $\text{End}(L_\alpha)=k$, the restriction of the differential $d$ to $C_\alpha^i$ has the form $d_{\alpha, i}=1_{L_\alpha}\otimes\delta_{\alpha,i}$, for some $k$-linear map
	\[
	\delta_{\alpha,i}: \bigwedge^iV\longrightarrow\bigwedge^{i+1}V.
	\]
	The Cartier isomorphism in \eqref{carisom} gives
	\begin{equation} \label{e313}
		Z^i_X/B^i_X=	\mathcal H^i(F_*\Omega_X^\bullet) \cong \Omega_X^i \cong \mathcal O_X^{\oplus\binom ni}.
	\end{equation}
	On the other hand, from the decomposition \eqref{e312}, it follows that the cohomology is of the form
	\[
	\mathcal H^i(F_*\Omega_X^\bullet) \cong \bigoplus_{\alpha\in\Gamma} \mathcal H^i(C_\alpha^\bullet) \cong 
	\bigoplus_{\alpha\in\Gamma} L_\alpha^{\oplus h_{\alpha,i}},
	\]
	for suitable integers $h_{\alpha,i}\geq0$. Comparing this with \eqref{e313}, no nontrivial $L_\alpha$ can occur. Hence
	\begin{equation} \label{e316}
		\mathcal H^i(C_\alpha^\bullet)=0 \,\,\text{for}\,\,\alpha\neq0, \qquad \mathcal H^i(C_0^\bullet) \cong 	\mathcal O_X^{\oplus\binom ni}.
	\end{equation}
	Since $C_0^i= \cO_X \otimes_k\bigwedge^iV$, $\text{rank}\, (C_0^i)\,=\, \binom ni$. Thus, \eqref{e316} forces every differential in $C_0^\bullet$ to be zero. Hence, the entire trivial summand $\mathcal O_X^{\oplus\binom ni}$ is contained in $Z_X^i$. Moreover, since $\mathcal H^i(C_\alpha^\bullet)=0$ for every $\alpha\neq0$ by \eqref{e316}, the complex $C_\alpha^\bullet$ is exact. Therefore 
	$$B^i_X	\cap \bigoplus_{ 0 \ne \alpha\in\Gamma}C_\alpha^i =Z^i_X	\cap \bigoplus_{ 0 \ne \alpha\in\Gamma}C_\alpha^i.$$
	Thus, $Z^i_X =B^i_X \oplus C_0^i \cong B^i_X \oplus \mathcal O_X^{\oplus\binom ni}$. 
	
	Fix $\alpha\neq0$. Let $r_i$ be the rank of the image in degree $i$:
	\[
	r_i=\text{rk}\, \text{im}\bigg(\delta_{\alpha,i-1}:\bigwedge^{i-1}V\longrightarrow\bigwedge^iV\bigg).
	\]
	Set $r_0=0$. Exactness of $C_\alpha^\bullet$ gives $\ker \delta_{\alpha,i}\,=\,\im \delta_{\alpha,i-1}$. Thus, the rank-nullity theorem gives $r_i+r_{i+1}=\binom ni$. Substituting $i=0$, we get $r_1=1$. For $i=1$, we get $r_2=(n-1)$. By induction, it follows that $r_i=\binom{n-1}{i-1}$,
	for $1 \leq i \leq n$. For each $\alpha \ne 0$, the $L_\alpha$-part of $B_X^i$ is the image in degree $i$, hence has multiplicity $r_i$. This proves 
	\[
	B_X^i \cong \bigoplus_{\alpha\neq0} L_\alpha^{\oplus\binom{n-1}{i-1}}.
	\]
	Hence
	\[
	Z_X^i \cong \mathcal O_X^{\oplus\binom ni} \oplus \bigoplus_{\alpha\neq0} L_\alpha^{\oplus\binom{n-1}{i-1}}.
	\]
\end{proof}

\begin{lemma} \label{toracm}
	Let $X$ be an ordinary smooth projective variety of dimension $n$ with trivial cotangent bundle. Then the line bundles $L_{\alpha}$ for $\alpha \ne 0$ appearing in \eqref{e310} are ACM \wrt every very ample divisor $H$.
\end{lemma}

\begin{proof} 
	By Proposition \ref{abdecom} with $i=1$,
	$$ B_X^1\cong\bigoplus_{\alpha\neq0}L_\alpha. $$
	Since $X$ is ordinary, $H^j(X,B_X^1)=0$ for every $j\geq0$; see Definition \ref{ordinarydef}. Hence
	\begin{equation} \label{e322}
		H^j(X,L_\alpha)=0 
	\end{equation}
	for every $\alpha\neq0$ and every $j\geq0$. 
	Let $H$ be a very ample divisor. For $t>0$, the line bundle $L_\alpha(tH)$ is ample, since
	$ \bigl(L_\alpha(tH)\bigr)^{\otimes p} \cong\mathcal O_X(ptH)$.
	Since $X$ is Frobenius split, \cite[Corollary 1.12]{schwede2026frobenius} gives
	\begin{equation} \label{e5025}
		H^j(X,L_\alpha(tH))=0 
	\end{equation}
	for every $j>0$. For $t<0$, Serre duality and $\omega_X\cong\mathcal O_X$ give
	$$ H^j(X,L_\alpha(tH)) \cong H^{n-j}\bigl(X,L_\alpha^{-1}(-tH)\bigr)^\vee =0 $$
	for $1\leq j\leq n-1$, because $L_\alpha^{-1}(-tH)$ is ample. This proves that $L_\alpha$ is ACM with respect to $H$.
\end{proof}

\begin{lemma} \label{nonvanabelian}
	Let $X$ be an ordinary smooth projective variety of dimension $n$ with trivial cotangent bundle.  Then $H^q(X, \cO_X) \ne 0$ for $0 \le q \le n$.
\end{lemma}

\begin{proof}
	By \cite[Theorem 1]{mehta1987varieties}, there is a finite Galois \'{e}tale cover $\pi:A\longrightarrow X$
	of degree $p^m$ for some $m \ge 0$, where $A$ is an abelian variety and $\text{char} \,k=p$. The Galois group $G$ of $\pi$ is a finite $p$-group. Thus, the trivial module $k$ is the unique simple $k[G]$-module; see \cite[Proposition 1.34, p. 13]{benson2011representations}. Hence the regular representation of $G$, namely $k[G]$, admits a Jordan-H\"{o}lder filtration
	\[
	0=V_0\subset V_1\subset\cdots\subset V_{p^m}=k[G]
	\]
	with $V_\nu/V_{\nu-1}\cong k$ as $k[G]$-modules. Since $\pi:A \longrightarrow X$ is a $G$-torsor, for a $G$-representation $V$, the associated bundle on $X$ is $A\times^G V$. For the $G$-module $k[G]$, finite \'{e}tale descent gives  
	\[
	A\times^G k[G]\cong \pi_*\mathcal O_A.
	\]
	The subspaces $V_i\subset k[G]$ are $G$-stable. Therefore they define subbundles $E_i=A\times^G V_i\subset A\times^G k[G]$ on $X$.
	Hence we obtain a filtration of subbundles
	\begin{equation} \label{etalefiltration}
		0=E_0\subset E_1\subset\cdots\subset E_{p^m}=\pi_*\mathcal O_A,
	\end{equation}
	such that $E_i/E_{i-1} \simeq A\times^G(V_i/V_{i-1})$. Since $V_i/V_{i-1}\cong k$ with the trivial $G$-action, we get $A\times^G k=\mathcal O_X$. Hence $E_i/E_{i-1} \cong \cO_X$.
	
	If $H^q(X,\mathcal O_X)=0$ for some $q$ with $0 \le q \le n$, by using \eqref{etalefiltration}, we obtain $H^q(X,\pi_*\mathcal O_A)=0$. This gives $H^q(A,\mathcal O_A)=H^q(X,\pi_*\mathcal O_A)=0$, which contradicts $h^q(A,\mathcal O_A)=\binom nq>0$; see \cite[Chapter III, \S 13, Corollary 2, p. 129]{mumford1974abelian}.
\end{proof}

\begin{cor} \label{acmabelianresult}
	Let $X$ be an ordinary smooth projective variety of dimension $n$ with trivial cotangent bundle. Let $H$ be a very ample divisor.  Then, for every $1 \le i \le n$, the following hold:
	\begin{enumerate}
		\item [(1)]   $B_X^i$ is ACM \wrt $H$.
		
		\item [(2)] If $n\ge2$, every $Z_X^i$ is exactly $(n-1)$-away ACM.
		Its only intermediate cohomology occurs at twist zero, where the trivial summands contribute in every degree $1,\ldots,n-1$.
		
		\item [(3)] Let $n \ge 2$. For every integer $e\ge1$ and every $a \in \bZ$, $F_*^e(Z_X^i(aH))$ is ACM if and only if  $p^e\nmid a$.
	\end{enumerate}
\end{cor}

\begin{proof}
	For $t\in\mathbb Z$ and $0<j<n$, Proposition \ref{abdecom} and Lemma \ref{toracm} give
	\[
	H^j(X,B_X^i(tH))=0, \qquad H^j(X,Z_X^i(tH)) \cong H^j(X,\mathcal O_X(tH))^{\oplus\binom ni}.
	\]
	The latter group vanishes for $t\ne0$ by \cite[Corollary 1.12]{schwede2026frobenius} and Serre duality, and is nonzero for $t=0$ by Lemma \ref{nonvanabelian}. This proves (1) and (2).

	The projection formula gives
	\[
	H^j\!\left(X,F_*^e(Z_X^i(aH))(tH)\right) \cong H^j\!\left(X,Z_X^i((a+p^et)H)\right).
	\]
	Hence, for $n\ge2$, $F_*^e(Z_X^i(aH))$ is ACM precisely when $a+p^et\ne0$ for every $t\in\mathbb Z$, equivalently $p^e\nmid a$.
\end{proof}

\begin{lemma} \label{cohotorline}
	Let $X$ be an ordinary smooth projective variety of dimension $n$ with trivial cotangent bundle. Let $H$ be a very ample divisor. For $\alpha \ne 0$, let $L_{\alpha}$ be the line bundle appearing in \eqref{e310}. Then 
	\begin{enumerate}
		\item [(1)] $H^j(X,L_\alpha(tH))=0 \,\,\text{for all} \,\,j>0$ and $t>0$.
		\item [(2)] $H^0(X,L_\alpha(tH))=0$ if and only if $t \le 0$.
		\item [(3)] $H^n(X,L_\alpha(tH))=0$ if and only if $t\geq0$.
	\end{enumerate}
\end{lemma}

\begin{proof}
	The statement $(1)$ follows from the proof of Lemma \ref{toracm}; see \eqref{e5025}. 
	
	Since $L_\alpha$ is numerically trivial, Riemann–Roch gives
	\begin{equation} \label{e325}
		\chi\bigl(X,L_\alpha(tH)\bigr)=\chi\bigl(X,\mathcal O_X(tH)\bigr).
	\end{equation}
	For $t>0$, the line bundles $L_{\alpha}(tH)$ and $\cO_X(tH)$ are ample. Thus, for $t>0$ and  $j>0$,  $H^j(X, \cO_X(tH))=H^j(X, L_{\alpha}(tH))=0$ by \cite[Corollary 1.12]{schwede2026frobenius}. Hence, using \eqref{e325}, for all $t>0$, we obtain
	\[
	h^0\bigl(X,L_\alpha(tH)\bigr)=h^0\bigl(X,\mathcal O_X(tH)\bigr)>0.
	\]
	For $t<0$, $H^0(X,L_\alpha(tH))=0$, since $L_\alpha(tH)$ is anti-ample. Since $L_{\alpha}$ is a nontrivial torsion line bundle, $H^0(X, L_{\alpha})=0$; see \eqref{e322}. Thus, $H^0(X,L_\alpha(tH))=0$ if and only if $t \le 0$. This proves $(2)$.
	
	Since $\omega_X\cong\mathcal O_X$, Serre duality gives
	\begin{equation} \label{e5026}
		H^n(X,L_\alpha(tH)) \cong H^0(X,L_\alpha^{-1}(-tH))^\vee. 
	\end{equation}
	Since $L_\alpha^{-1}\cong L_{-\alpha}$ is again a nontrivial line bundle appearing in \eqref{e310}, statement $(2)$ shows that the group in \eqref{e5026} vanishes if and only if $t\geq0$. This proves $(3)$.
\end{proof}

\begin{cor} \label{almostulrichB^i_X}
	Let $X$ be an ordinary smooth projective variety of dimension $n$ with trivial cotangent bundle. Let $H$ be a very ample divisor. Then for $1 \le i \le n$, $s, a \in \bZ$, and $e \ge 1$, the following hold:
	\begin{enumerate}
		\item [(1)] $B_X^i(sH)$ is almost Ulrich if and only if $n-1\leq s\leq2$. In particular, for $n \ge 4$, none of $B_X^i(sH)$ are almost Ulrich.
		\item [(2)] $F^e_*\bigl(Z_X^i(aH)\bigr)$ is almost Ulrich if and only if  $p^e(n-1)<a<2p^e$. Consequently, for $n\geq2$, $F^e_*\bigl(Z_X^i(aH)\bigr)$  is almost Ulrich if and only if $n=2$ and $p^e<a<2p^e$.
	\end{enumerate} 
\end{cor}

\begin{proof}
	By Proposition \ref{abdecom}, Lemma \ref{cohotorline}, and Serre duality, we have
	\[
	\begin{array}{c|cc}
		E & H^0(X,E(tH))=0 & H^n(X,E(tH))=0\\ \hline B_X^i & t\le0 & t\ge0\\
		Z_X^i & t<0 & t>0
	\end{array}
	\]
	Since $B_X^i$ is ACM, the almost-Ulrich conditions for $B_X^i(sH)$ in Definition \ref{almostul} reduce to $H^n\bigl(X, B_X^i((s-m)H)\bigr)=0$ for $m=n-1$ and $H^0\bigl(X, B_X^i((s-m)H)\bigr)=0$ for $m=2$. These vanishings are equivalent to $s-(n-1)\ge0$ and $s-2\le0$. This proves (1).
	
	By the projection formula, we obtain
	\footnotesize
	\begin{equation*}
		H^j\left(X, F^e_*\bigl(Z_X^i(aH)\bigr)\otimes\mathcal O_X(-mH)\right) \cong H^j\left(X, F^e_* \bigg(\bigl(Z_X^i(aH)\bigr)\otimes (F^e)^* \mathcal O_X(-mH)\bigg) \right) \cong H^j(X, Z^i_X((a-p^em)H)).
	\end{equation*}
	\normalsize
	Hence, the corresponding top and zeroth-almost Ulrich conditions for $F_*^e(Z_X^i(aH))$ are $H^n(X, Z^i_X((a-p^em)H))=0$ for $m= n-1$ and $H^0(X, Z^i_X((a-p^em)H))=0$ for $m=2$. These vanishings are equivalent to
	\[
	a-p^e(n-1)>0, \qquad a-2p^e<0.
	\]
	For $n\ge2$, these inequalities force $n=2$ and $p^e<a<2p^e$, hence $p^e \nmid a$. Thus, by Corollary \ref{acmabelianresult}, $F_*^e(Z_X^i(aH))$  is ACM. For $n=1$, the ACM condition is vacuous. This proves (2).
\end{proof}

\begin{remark}
	Since an abelian variety $A$ has $\Omega^1_A$ trivial, the above result applies to ordinary abelian varieties.
\end{remark}

	\section{Classification of varieties with Ulrich twists of $B^i_X$ and $Z^i_X$}
			\label{four}
			
		We work over an algebraically closed field $k$ of characteristic $p>0$. Let $X$ be a smooth projective variety of dimension $n$ over $k$, and let $H$ be a very ample divisor on $X$. In this section, we investigate when $B_X^i(aH)$ or $Z_X^i(aH)$ is Ulrich with respect to $H$ for some integer $a$.


	For $e \ge 1$, recall the vector bundles:
		\begin{equation} \label{e40}
			B_{e}:=\operatorname{coker} \bigl(\mathcal O_X\longrightarrow F_*^e\mathcal O_X\bigr).
		\end{equation}
		When $e=1$, $B_{1}$ is the sheaf of locally exact  $1$-forms $B^1_X$.  We denote $q=p^e$. Twisting \eqref{e40} by $\cO_X(sH)$ for $s \in \bZ$ gives
		\begin{equation} \label{ee42}
			0 \longrightarrow\mathcal O_X(sH) \longrightarrow F_*^e\mathcal O_X(sqH) \longrightarrow B_{e}(sH) \longrightarrow 0.
		\end{equation}
		
		The classification problem for curves for which $B_e(aH)$ is Ulrich is addressed in Section \ref{curclassi}. We therefore restrict attention to $n\ge2$ when studying $B_X^i$. In the following classification result, $X$ is not assumed to be Frobenius split.
		
		\begin{theorem} \label{classisur}
			Let $X$ be a smooth projective surface over $k$, and let $H$ be a very ample divisor. Let $e \ge 1$ and $a \in \bZ$. Then $B_{e}(aH)$ 
			is Ulrich with respect to $H$ if and only if $a=1, e=1, p=2, H^2=1$,
			and $(X,H)\cong\bigl(\mathbb P^2,\mathcal O_{\mathbb P^2}(1)\bigr)$.
		\end{theorem}
		
		\begin{proof}
			We denote $q=p^e$, and $d=H^2$. 
			Since $\chi(F_*^e\mathcal O_X(sqH))=\chi(\mathcal O_X(sqH))$, from \eqref{ee42}, we obtain
			\[
			\chi\bigl(B_{e}(sH)\bigr)=\chi\bigl(\mathcal O_X(qsH)\bigr)-\chi\bigl(\mathcal O_X(sH)\bigr).
			\]
			By the Riemann-Roch theorem for a surface \cite[Example 4.1.2, p. 433]{hartshorne1977algebraic}, we obtain
			\begin{eqnarray} \label{e41}
				\chi\bigl(B_{e}(sH)\bigr)&=& \frac{1}{2}\left[qsH \cdot \big(qsH-K_X\big)\right]-\frac{1}{2}\left[sH \cdot \big(sH-K_X\big)\right]\nonumber\\
				&=& \frac{1}{2}\left[(q^2-1)s^2d-(q-1)sK_X\cdot H\right].
			\end{eqnarray}
			
			Suppose that $B_{e}(aH)$ is Ulrich. Then the Ulrich vanishings give $H^i(X, B_{e}(sH))=0$ for $s=a-1,a-2$ and for all $i$. Thus, the quadratic polynomial in \eqref{e41} vanishes at $s=a-1$ and $s=a-2$. Since $\chi(B_{e})=0$ by \eqref{ee42}, $0$ is one of its roots. Thus, we obtain $a=1$ or $a=2$.
			
			The case $a=2$ is impossible. Indeed, if $B_{e}(2H)$ were Ulrich, then it would be initialized, giving $H^0\bigl(X,B_{e}(H)\bigr)=0$. This contradicts Corollary \ref{oneformnonvan}.
			
			Therefore, $a=1$.  Since $s=a-2=-1$ is a root of  \eqref{e41}, substituting $s=-1$ gives
			\[
			K_X\cdot H=-(q+1)d.
			\]
			Let $C\in |H|$ be a smooth curve. By the adjunction formula, using $q=p^e\geq 2$ and $d\geq 1$, we get
			\begin{equation} \label{e45}
				-2 \le 2g(C)-2=(K_X+H)\cdot H=-(q+1)d+d=-qd \le -2.
			\end{equation}
			Hence $qd=2$, so $q=2$ and $d=1$. Consequently, $p=2$, $e=1$. Since $X$ is a nondegenerate projective surface of degree $d=1$, $X$ must be a linear plane. Therefore, $(X,H)\cong
			\bigl(\mathbb P^2,\mathcal O_{\mathbb P^2}(1)\bigr)$.
			
			Conversely, in characteristic $p=2$,
			\[
			F_*\mathcal O_{\mathbb P^2}\cong \mathcal O_{\mathbb P^2} \oplus \mathcal O_{\mathbb P^2}(-1)^{\oplus 3},
			\]
			and hence $B_{1}(1) \cong \mathcal O_{\mathbb P^2}^{\oplus 3}$. Thus, $B^1_{\bP^2}(1)=B_{1}(1)$ is Ulrich \wrt $\cO_{\bP^2}(1)$.
		\end{proof}

		\begin{cor} \label{surfaceB2}
			Let $X$ be a smooth projective surface over $k$, and let $H$ be a very ample divisor. Let $a\in\mathbb Z$. Then $B^2_X(aH)$ 
			is Ulrich with respect to $H$ if and only if 
			$a=2, p=2, H^2=1$, and $(X,H)\cong \bigl(\mathbb P^2,\mathcal O_{\mathbb P^2}(1)\bigr)$.
		\end{cor}
		
		\begin{proof}
			Since $\dim X=2$, by the duality in \eqref{cartierduality}, we get $B_X^1\cong (B_X^2)^\vee\otimes\omega_X$.
			Thus, if $B^2_X(aH)$ is Ulrich \wrt $H$, by Lemma \ref{ulrichdual}, its Ulrich dual is 
			\begin{equation}
				\big(B^2_X(aH)\big)^{\vee} \otimes \omega_X(3H) \cong \big(B^2_X\big)^{\vee} \otimes \omega_X((3-a)H) \cong B^1_X \otimes \cO_X((3-a)H).
			\end{equation}
			Therefore, $B^2_X(aH)$ is Ulrich \wrt $H$ if and only if $B^1_X((3-a)H)$ is Ulrich \wrt $H$. The assertion then follows from Theorem \ref{classisur} with $e=1$.
		\end{proof}

		\begin{remark}
			Let $\text{char}(k)$ be $p=2$. By duality in \eqref{cartierduality}, we obtain
			\[
			B_{\mathbb P^2}^2(2) \cong (B_{\mathbb P^2}^1)^\vee \otimes\omega_{\mathbb P^2}(2) \cong \mathcal O_{\mathbb P^2}(1)^{\oplus 3}
			\otimes\mathcal O_{\mathbb P^2}(-1) \cong \mathcal O_{\mathbb P^2}^{\oplus 3}.
			\]
		\end{remark}
		
		We now study whether $B^i_X(aH)$ is Ulrich when $\dim X \ge 3$.
		
		\begin{proposition} \label{P3example}
			Assume that $\text{char}(k)$ is $p=2$. Then $B_{\bP^3}^2(2)$ is Ulrich \wrt $\cO_{\bP^3}(1)$.
		\end{proposition}
		
		\begin{proof}
			Denote $E:=B_{\mathbb P^3}^2(2)$, and $\mathcal G:=(F_*\Omega_{\mathbb P^3}^1)(2)$.
			Using the Euler sequence
			$$0 \longrightarrow \Omega^1_{\bP^3} \longrightarrow \cO_{\bP^3}(-1)^{\oplus 4} \longrightarrow \cO_{\bP^3} \longrightarrow 0,$$
			we obtain
			\begin{equation} \label{laste31}
				H^1(\mathbb P^3,\Omega_{\mathbb P^3}^1(2))=H^2(\mathbb P^3,\Omega_{\mathbb P^3}^1)=H^3(\mathbb P^3,\Omega_{\mathbb P^3}^1(-2))=0.
			\end{equation}
			Since $p=2$, the projection formula identifies $\cG(j) \cong (F_*\Omega_{\mathbb P^3}^1)(2+j) \cong F_*(\Omega_{\mathbb P^3}^1(4+2j))$. Thus, \eqref{laste31} gives the vanishings
			\[
			H^j(\mathbb P^3,\mathcal G(-j))=0, \qquad j=1,2,3.
			\]
			Thus $\mathcal G$ is $0$-regular and hence globally generated by \cite[Theorem 1.8.3, p. 99]{lazarsfeld2004positivity}. The defining sequence
			\[
			0\longrightarrow Z_{\mathbb P^3}^1(2) \longrightarrow\mathcal G \longrightarrow E \longrightarrow0
			\]
			shows that $E$ is globally generated.
			Using duality in \eqref{cartierduality} and $\omega_{\mathbb P^3}\cong\mathcal O_{\mathbb P^3}(-4)$, we obtain
			\[
			E^\vee=(B_{\mathbb P^3}^2)^\vee(-2)\cong B_{\mathbb P^3}^2(2) =E.
			\]
			Consequently, $(\det E)^{\otimes2}\cong\mathcal O_{\mathbb P^3}$.
			Choose $r=\operatorname{rk}E$ global sections $s_1,\ldots,s_r$  of $E$ forming a basis at one point, and let $\sigma$ be their determinant. Under the preceding trivialization, $\sigma^{\otimes2}$ is a nonzero constant. The sections define a morphism $\varphi:\mathcal O_{\mathbb P^3}^{\oplus r}\longrightarrow E$. Its determinant is the nowhere-vanishing section $\sigma$. Hence $\varphi$ is locally represented by a square matrix with invertible determinant and is therefore an isomorphism. 
			By  \eqref{ee523}, $\text{rk}(B^2_{\bP^3})=r=2(2^3-1)=14$. Therefore
			\[
			B_{\mathbb P^3}^2(2) =E \cong\mathcal O_{\mathbb P^3}^{\oplus14},
			\]
			which is Ulrich with respect to $\mathcal O_{\mathbb P^3}(1)$.
		\end{proof}

		\begin{theorem} \label{B_X^i(aH)ulthm}
			Let $X$ be a smooth projective variety of dimension $n\geq3$, and let $H$ be a very ample divisor. Let $1\le i\le n$ and $a\in\mathbb Z$.
			If $B_X^i(aH)$ is Ulrich, then necessarily
			\[
			n=3,\qquad p=2,\qquad i=2,\qquad a=2.
			\]
		\end{theorem}

		\begin{proof}
			Suppose that $B_X^i(aH)$ is Ulrich. By duality \eqref{cartierduality}, we have $(B_X^i)^\vee\otimes\omega_X \cong
			B_X^{n+1-i}$. Hence by Lemma \ref{ulrichdual}, the Ulrich dual of $B_X^i(aH)$ is given by
			\[
			E:=B_X^i(aH)^{\vee} \otimes \omega_X \otimes \cO_X((n+1)H) \cong B_X^{n+1-i}\bigl((n+1-a)H\bigr).
			\]
			Since $B_X^i(aH)$ is Ulrich, $H^0\bigl(X,B_X^i((a-1)H)\bigr)=0$ by Definition \ref{uldefin}. 
			By Proposition \ref{nonvanishingexactforms}, this requires
			\begin{equation} \label{ee51}
				p(a-1) \le i.
			\end{equation}
			Similarly, applying Proposition \ref{nonvanishingexactforms} to the Ulrich dual 
			$E=B_X^{n+1-i}((n+1-a)H)$ gives
			\begin{equation} \label{ee52}
				p(n-a) \le n+1-i.
			\end{equation}
			Adding the inequalities \eqref{ee51} and \eqref{ee52} yields
			\[
			2(n-1)\leq p(n-1)\leq n+1.
			\]
			Since $n\geq3$, we obtain $n=3$ and $p=2$. The inequalities \eqref{ee51} and \eqref{ee52} involving $a$ and $i$ are then equalities, so
			\[
			i=2(a-1).
			\]
			Since $1\leq i\leq3$, this forces $i=a=2$.
		\end{proof}
		
		The evidence suggests that if $B^2_X(2H)$ is Ulrich, then one should have $(X, H) \cong (\bP^3, \cO_{\bP^3}(1))$. Lemma \ref{P^3lemma} isolates a precise obstruction to this implication. 
		
		By \cite[Lemma 5.1]{li2013instability}, the rank of $B^i_X$ is given by
		\begin{equation} \label{ee523}
			\text{rk}(B^i_X)=\binom{n-1}{i-1}(p^n-1).
		\end{equation}
		Let $C$ be a smooth complete-intersection curve cut out by $n-1$ general members of $|H|$. By adjunction formula, we have $\big(K_X+(n-1)H\big) \cdot C=2g(C)-2$. Denote $d:=H^n$, and $g:=g(C)$. This gives
		\begin{equation} \label{e462}
			K_X \cdot H^{n-1}= 2g-2-(n-1)H^n=2g-2-(n-1)d.
		\end{equation}

		\begin{lemma} \label{P^3lemma}
			Assume that $n\ge3$ and that $B_X^i(aH)$ is Ulrich with respect to $H$. Then $K_X+4H$ is nef. Moreover, $(K_X+4H) \cdot H^2=0$ if and only if $(X, H) \cong (\bP^3, \cO_{\bP^3}(1))$.
		\end{lemma}
		\begin{proof}
			Since $B^i_X(aH)$ is Ulrich, by Theorem \ref{B_X^i(aH)ulthm}, $n=3, p=2, i=2, a=2$. By the duality in \eqref{cartierduality}, we have $B_X^2\simeq (B_X^2)^\vee\otimes \omega_X$.
			Consequently, for $E=B^2_X(2H)$, we obtain
			\begin{equation} \label{ee524}
				E\simeq E^\vee(K_X+4H).
			\end{equation}
			Since $\text{rk}(E)=14$ by \eqref{ee523}, taking determinants in \eqref{ee524} gives $(\det E)^{\otimes2} \cong \mathcal O_X\bigl(14(K_X+4H)\bigr)$. Since $E$ is Ulrich, it is globally generated, and hence so is $\det E$; see \cite[\href{https://stacks.math.columbia.edu/tag/01CJ}{Lemma 01CJ}]{stacks-project}. Consequently, $\mathcal O_X\bigl(14(K_X+4H)\bigr)$ is globally generated. It follows that $K_X+4H$ is nef.
			
			By \eqref{e462}, we obtain
			\begin{equation} \label{ee530}
				(K_X+4H)\cdot H^2=2(g+d-1).
			\end{equation}
			Since $g\ge0$ and $d\ge1$, \eqref{ee530} gives $(K_X+4H)\cdot H^2=0$ if and only if $g=0,\, d=1$.
			Once $d=1$, the embedding defined by $H$ has degree one. Thus $X$ must be a linear subspace. Hence $(X,H)\simeq
			\bigl(\mathbb P^3,\mathcal O_{\mathbb P^3}(1)\bigr)$.
		\end{proof}

		In the following, we show that $Z_X^i(aH)$ is never Ulrich with respect to any very ample divisor $H$ on $X$.
		
		\begin{theorem} \label{closedformsnotulrich}
			Let $X$ be a smooth projective variety of dimension $n$. For every $1\leq i\leq n$ and every $a\in\mathbb Z$, $Z_X^i(aH)$ is not Ulrich \wrt any very ample divisor $H$.
		\end{theorem}
		
		\begin{proof}
			Let $H$ be a very ample divisor on $X$. Arguing by contradiction, suppose that $E:=Z_X^i(aH)$ is Ulrich \wrt $H$. Let $C$ be a general smooth complete intersection of $n-1$ members of $|H|$. We take $C=X$ if $n=1$. Denote $h=H|_C$. The conormal sequence is
			\[
			0\longrightarrow\mathcal O_C(-h)^{\oplus(n-1)} \longrightarrow\Omega_X^1|_C \longrightarrow\omega_C \longrightarrow0.
			\]
			Put $N=\cO_C(-h)^{\oplus(n-1)}$. Since $\omega_C$ has rank one, the exterior-power filtration
			associated with the conormal sequence
			\cite[Chapter~II, Exercise~5.16(d)]{hartshorne1977algebraic}
			gives a short exact sequence
			\begin{equation} \label{laste411}
				0\longrightarrow\bigwedge^iN \longrightarrow\Omega_X^i|_C \longrightarrow\bigwedge^{i-1}N\otimes\omega_C \longrightarrow0.
			\end{equation}
			Now $\bigwedge^{i-1}N \cong \mathcal O_C(-(i-1)h)^{\oplus\binom{n-1}{i-1}}$.
			We have the surjection $Z_X^i\twoheadrightarrow\Omega_X^i$ from \eqref{cartierseq}. Since restriction preserves surjections, from \eqref{laste411}, we obtain the surjection
			\begin{equation} \label{laste412}
				E|_C(-h)=Z_X^i|_C((a-1)h) \twoheadrightarrow\Omega_X^i|_C((a-1)h) \twoheadrightarrow
				\bigl(\omega_C((a-i)h)\bigr)^{\oplus\binom{n-1}{i-1}}.
			\end{equation}
			
			Since $E|_C$ is Ulrich \wrt $h$ by Lemma \ref{ulrichdual}, we have
			\[
			H^1(C,E|_C(-h))=0.
			\]
			Since $C$ is a curve, taking the \les in cohomology for \eqref{laste412}, we get 
			\[
			H^1(C, \bigl(\omega_C((a-i)h)\bigr))=0.
			\]
			Serre duality therefore gives $H^0\bigl(C,\mathcal O_C((i-a)h)\bigr)=0$. Since $h$ is very ample, this forces $a\geq i+1$. Consequently,
			\[
			p(a-1)\geq pi\geq i+1.
			\]
			Since $B_X^i \subset Z^i_X$, Proposition \ref{nonvanishingexactforms} gives
			\[
			0\neq H^0\bigl(X,B_X^i((a-1)H)\bigr) \subseteq H^0(X,E(-H))=0,
			\]
			a contradiction.
		\end{proof}

	\section{Classification of abelian varieties with Ulrich $B^1_X \otimes L$}
	\label{five}
	
	Throughout, we work over an algebraically closed field $k$ of characteristic $p>0$. Let $A$ be an abelian variety of dimension $n$. In this section, we prove that $A$ is isogenous to a product of $n$ elliptic curves if and only if there exist a very ample line bundle $H$ and a line bundle $L$ on $A$ such that $B_A^1\otimes L$ is Ulrich with respect to $H$.
		
Let $L$ be a line bundle on $A$, and let $H$ be a very ample divisor on $A$. We consider the line bundle
\[
M_j:=L\otimes H^{-j},\qquad 1\leq j\leq n.
\]
Tensoring \eqref{Fsplitseq} with $M_j$ gives the \ses
\[
0\longrightarrow M_j \longrightarrow F_*M_j^{\otimes p} \longrightarrow B_A^1\otimes M_j \longrightarrow0.
\]
Therefore, by condition $D_1$ of Definition \ref{ulprop}, we obtain
\begin{cor} \label{lcor1}
	The sheaf $B_A^1\otimes L$ is Ulrich \wrt $H$ if and only if $H^i(A, M_j) \longrightarrow H^i(A, M_j^{\otimes p})$
	is bijective for every $i$ and $1 \le j \le n$.
\end{cor}
In particular, it suffices that both $M_j$ and $M_j^{\otimes p}$ are acyclic for every $1 \le j \le n$.

\begin{remark}[Abelian varieties of dimension $1$]
	\label{sec5:rem1}
	Let $E$ be an elliptic curve and put $M=L\otimes H^{-1}$.
	Then $B_E^1\otimes L$ is Ulrich with respect to $H$ if and only if $\deg M=0$ and either $M^{\otimes p}\not\cong\mathcal O_E$, or both $M \cong\mathcal O_E$ and $E$ is ordinary. 
	
	Indeed, if $B_E^1\otimes L$ is Ulrich, by Corollary \ref{lcor1}, $\chi(M)=\chi(M^{\otimes p})$. Then Riemann-Roch gives $\deg M=0$. Now if $M^{\otimes p} \cong\mathcal O_E$, then Corollary \ref{lcor1} gives $h^0(E, M)=1$. Thus, $M \cong \cO_E$. Moreover, when $M \cong \mathcal O_E$ the Frobenius maps on $H^i(E, \cO_E)$ are bijective, and that is precisely ordinarity.
	
	Conversely, if $\deg M=0$ and $M^{\otimes p}\not\cong\mathcal O_E$, then both $M$ and $M^{\otimes p}$ are nontrivial degree-zero line bundles and hence acyclic, so Corollary \ref{lcor1} applies.
\end{remark}

Let $A$ be an abelian variety of dimension $n$. Let $M$ be a line bundle on $A$. Let $t_x: A \longrightarrow A$ be the translation  by $x \in A$. Then we define
\[
\phi_M:A\longrightarrow\widehat A, \qquad x \longmapsto t_x^*M \otimes M^{-1}.
\]
Under the canonical identification $\varpi:A\simeq\widehat{\widehat A}$, $\phi_M$ is symmetric by \cite[Proposition 7.4.3, p. 69]{conrad2015abelian}; that is,
\begin{equation} \label{lastee6022}
	\phi_M^\vee \circ \varpi=\phi_M.
\end{equation}

\begin{lemma} \label{lasrkeylem1}
	Let $A$ have dimension $n\geq2$, and let $M$ be a line bundle such that
	\[
	K:=\bigl(\ker\phi_M\bigr)^0_{\mathrm{red}}
	\]
	is an elliptic curve. Put $\eta=M|_K$. Assume that $\eta^{\otimes p}\not\cong\mathcal O_K$. Then $H^i(A,M) = H^i(A,M^{\otimes p})=0$ for every $i$.
\end{lemma}

\begin{proof}
	Let
	\[
	q:A\longrightarrow B:=A/K
	\]
	be the quotient. Let $b \in B(k)$ and choose $x \in A(k)$ such that $q(x)=b$. Since $\ker q=K$, the fiber is $A_b=q^{-1}\{b\}=x+K$. Let $\iota:K\hookrightarrow A$ be the inclusion. Since $\phi_\eta=\iota^\vee\circ\phi_M\circ\iota=0$,
	the line bundle $\eta$ has degree zero. The translation $t_x: A \longrightarrow A$ restricts to an isomorphism
	\[
	t_x \vert_{K}:K\xrightarrow{\sim}A_b, \qquad y\longmapsto x+\iota(y).
	\]
	Its composite with the inclusion $A_b\hookrightarrow A$ is $t_x\circ\iota$. Consequently,
	\begin{equation} \label{e509}
		(t_x \vert_{K})^*(M|_{A_b})=\iota^*t_x^*M \cong t_x^*M \vert_{K}.
	\end{equation}
	
	Now $t_x^*M\otimes M^{-1}=\phi_M(x)\in \text{Pic}^0(A)$.
	Since $\phi_M^{\vee} \circ \varpi=\phi_M$ by \eqref{lastee6022}, dualizing $\phi_M \circ \iota=0$ gives
	\[
	\iota^\vee\circ\phi_M=(\phi_M\circ\iota)^\vee\circ\varpi=0.
	\]
	Thus $\phi_M(x) \vert_{K} \cong \cO_K$. Hence, by \eqref{e509}, we have 
	\begin{equation} \label{e510}
		(t_x \vert_{K})^*(M|_{A_b}) \cong M \vert_{K}=\eta.
	\end{equation}
	Similarly, we obtain
	\begin{equation} \label{e511}
		(t_x \vert_{K})^*\!\left(M^{\otimes p}|_{A_b}\right) \cong\eta^{\otimes p}.
	\end{equation}
	
	Since $\deg \eta =0$ and $\eta^{\otimes p}$ is nontrivial, both $\eta$ and $\eta^{\otimes p}$ are acyclic on $K$, that is, $H^{r}(K, \eta)=H^{r}(K, \eta^{\otimes p})=0$ for $r=0, 1$. Thus, by \eqref{e510} and \eqref{e511}, we get $H^{r}(A_b, M|_{A_b})=H^{r}(A_b, M^{\otimes p}|_{A_b})=0$ for $r=0, 1$. Therefore, by the base change theorem \cite[Theorem 4.2(d), p. 21]{milneAV}, for all $r \ge 0$, we have
	\[
	R^rq_*M=R^rq_*(M^{\otimes p})=0.
	\]
	Therefore, by the Leray spectral sequence, for every $i$, we obtain $H^i(A,M)=H^i(A,M^{\otimes p})=0$.
\end{proof}

\begin{theorem} \label{sec5:thm1}
	Let $A$ be an abelian variety of dimension $n\geq2$, let $H$ be a very ample divisor on $X$, and let $L$ be a line bundle on $X$. Define
	\[
	M_j=L\otimes H^{-j}, \qquad K_j=(\ker\phi_{M_j})^0_{\mathrm{red}}, \qquad 1\leq j\leq n.
	\]
	Assume $B_A^1\otimes L$ is Ulrich with respect to $H$. Then we have
	\begin{equation} \label{laste601}
		(L-jH)^n=0 \qquad(1\leq j\leq n).
	\end{equation}
	Moreover, conditions \eqref{laste601} imply that every $K_j$ is an elliptic curve and that the addition map
	\[
	K_1\times\cdots\times K_n\longrightarrow A
	\]
	is an isogeny.
\end{theorem}

\begin{proof}
	Suppose first that $B_A^1\otimes L$ is Ulrich. By Corollary \ref{lcor1}, we obtain
	\begin{equation} \label{e500}
		\chi(A,M_j)=\chi(A,M_j^{\otimes p}).
	\end{equation}
	Riemann–Roch on an abelian variety gives \cite[Theorem 11.1, p. 54]{milneAV}
	\[
	\chi(A,M)=\frac{c_1(M)^n}{n!}.
	\]
	This gives $\chi(A,M^{\otimes p})=\frac{p^nc_1(M)^n}{n!}$. Therefore, by \eqref{e500}, we get
	\[
	(p^n-1)(L-jH)^n=0,
	\]
	which gives the numerical conditions $(L-jH)^n=0$. 
	
	We next show that they force the elliptic curve decomposition. By \cite[Theorem 11.1, p. 54]{milneAV}, the degree of $\phi_{M_j}$ is
	\[
	\deg\phi_{M_j}=\chi(A, M_j)^2.
	\]
	Since $\chi(A,M_j)=0$, every $K_j$ has positive dimension. Indeed, since $\dim A=\dim\widehat A$, if $\ker \phi_{M_j}$ were finite,
	the homomorphism $\phi_{M_j}: A\to\widehat A$ would be an isogeny by \cite[Proposition 7.1, p. 32]{milneAV}.
	Since an isogeny has strictly positive degree, this contradicts  $\deg\phi_{M_j}=0$. Thus, $\ker \phi_{M_j}$ is of positive dimension. Since each connected component of $\ker \phi_{M_j}$ is a translation of its identity component $(\ker\phi_{M_j})^0$, we conclude that every $K_j$ has positive dimension.
	
	Let $\iota_j:K_j\hookrightarrow A$ be the inclusion. Since $K_j \subset \ker \phi_{M_j}$, the identity $\phi_{M_j}=\phi_L-j\phi_H$ gives $(\phi_L-j\phi_H)\iota_j=0$, and hence
	\begin{equation} \label{e5021}
		\phi_L\iota_j=j\phi_H\iota_j.
	\end{equation}
	Dualizing \eqref{e5021} with $r$ in place of $j$, and using \eqref{lastee6022}, we obtain
	\begin{equation} \label{e5022}
		\iota_r^\vee\phi_L=(\phi_L\iota_r)^\vee \circ \varpi=(r\phi_H\iota_r)^\vee \circ \varpi=r\,\iota_r^\vee\phi_H.
	\end{equation}
	Both sides of \eqref{e5022} are homomorphisms $A\to\widehat K_r$. Now composing \eqref{e5021} with $\iota_r^\vee$ and using \eqref{e5022},  we obtain
	\begin{equation} \label{lastee701}
		{j\,\iota_r^\vee\phi_H\iota_j =\iota_r^\vee\phi_L\iota_j}={r\,\iota_r^\vee\phi_H\iota_j}.
	\end{equation}
	Thus $(j-r)\bigl(\iota_r^\vee\phi_H\iota_j\bigr)=0$. For $r\ne j$, we have $j-r\ne0$. Since $\Hom(K_j,\widehat K_r)$ is torsion-free as a $\bZ$-module by \cite[Lemma 10.6, p. 45]{milneAV}, we conclude
	\begin{equation} \label{e5023}
		\iota_r^\vee\phi_H\iota_j=0 \qquad(r\neq j).
	\end{equation}
	For the addition map $s:B:=\prod_{j=1}^nK_j\to A$, given by $(x_1, x_2, \cdots, x_n) \mapsto x_1+\cdots+x_n$, the dual is
	\[
	s^\vee:\widehat A \longrightarrow \widehat B=\prod_j\widehat K_j, \qquad P\longmapsto\bigl(\iota_j^*P\bigr)_j.
	\]
	Therefore, the $(r,j)$-entry of $s^\vee\phi_Hs:B\longrightarrow\widehat B$ is exactly
	\[
	K_j\xrightarrow{\iota_j}A \xrightarrow{\phi_H}\widehat A \xrightarrow{\iota_r^\vee}\widehat K_r.
	\]
	Thus, by \eqref{e5023}, every off-diagonal entry for $r \ne j$ vanishes. Moreover, the diagonal entries satisfy $\iota_j^\vee\phi_H\iota_j=\phi_{\iota_j^*H} =\phi_{H|_{K_j}}$. We therefore have
	\[
	s^\vee\phi_Hs=\operatorname{diag}\bigl(\phi_{H|_{K_1}},\ldots,\phi_{H|_{K_n}} \bigr).
	\]
	Each $H|_{K_j}$ is very ample, so the diagonal map on the RHS is an isogeny; see \cite[Proposition 8.1, Proposition 7.1]{milneAV}. Since $\ker s\subseteq\ker(s^\vee\phi_Hs)$, the map $s$ has finite kernel. Consequently,
	\[
	\sum_{j=1}^n\dim K_j\leq\dim A=n.
	\]
	Since $\dim K_j >0$, we obtain $\dim K_j=1$ for every $j$. Since each $K_j$ is a one-dimensional abelian subvariety of $A$, it is an elliptic curve. Moreover, since $\ker s$ is finite, the map $s:\prod_{j=1}^nK_j \longrightarrow A$ is an isogeny by \cite[Proposition 7.1(c)]{milneAV}.
\end{proof}

\begin{remark}  \label{sec5:rem2}
	Moreover, the proof of Theorem \ref{sec5:thm1} allows us to describe the pullbacks $s^*H$ and $s^*L$ as box tensor products of line bundles on the elliptic curves $K_j$, where $s:\prod_{j=1}^nK_j \longrightarrow A$ is the isogeny. Put $H_j:=H|_{K_j}, L_j:=L|_{K_j}$, and $\eta_j:=(L\otimes H^{-j})|_{K_j}$. 
	
	For $r \ne j$, by \eqref{e5023}, we have $\iota_r^\vee\phi_H\iota_j=0$. This gives $\iota_r^\vee\phi_L\iota_j=0$ for $r \ne j$ by \eqref{lastee701}. Consequently, the off-diagonal entries vanish for both $s^\vee\phi_Hs$ and $s^\vee\phi_Ls$, that is, 
	\begin{equation} \label{laste702}
		\phi_{s^*H}=\operatorname{diag}(\phi_{H_1},\ldots,\phi_{H_n}),\qquad \phi_{s^*L} =\operatorname{diag}(\phi_{L_1},\ldots,\phi_{L_n}).
	\end{equation}
	For $T=H$ or $L$, consider $Q_T:=s^*T\otimes\left(\boxtimes_{j=1}^{n}T|_{K_j}\right)^{-1}$.
Then \eqref{laste702} gives $\phi_{Q_T}=0$, hence $Q_T\in\operatorname{Pic}^0(B)$; see \cite[Definition of $\text{Pic}^0A$, p. 35-36]{milneAV}. 
	The restriction of $Q_T$ to each factor $K_j\subset B$ is trivial. Since the restriction map identifies
	\[
	\operatorname{Pic}^0(B) \cong \prod_{j=1}^{n}\operatorname{Pic}^0(K_j),
	\]
	we obtain $Q_T\cong\mathcal O_B$. Thus, $s^*H\cong\boxtimes_{j=1}^{n}H_j$, and  $s^*L\cong\boxtimes_{j=1}^nL_j$.
	
	Since $\eta_j \cong L_j \otimes H_j^{-j}$, \eqref{e5021} gives
	\[
	\phi_{\eta_j}=\phi_{L_j}-j\phi_{H_j}=\iota_j^\vee(\phi_L-j\phi_H)\iota_j=0.
	\]
	Thus $\eta_j\in\operatorname{Pic}^0(K_j)$, and
	$L_j=H_j^{\otimes j}\otimes\eta_j$. This gives  $s^*L\cong
	\boxtimes_{j=1}^{n}\bigl(H_j^{\otimes j}\otimes\eta_j\bigr)$ with $\eta_j\in\operatorname{Pic}^0(K_j)$.
	\end{remark}

\begin{cor} \label{sec5:cor1}
	Let  $A$ be an abelian variety of dimension $n\geq2$ over an algebraically closed field $k$ of characteristic $p>0$. There exist a very ample line bundle $H$ and a line bundle $L$ such that $B_A^1\otimes L$ is Ulrich with respect to $H$ if and only if $A$ is isogenous to a product of $n$ elliptic curves.
\end{cor}

\begin{proof}
	The necessary part follows from Theorem \ref{sec5:thm1}.
	Conversely, choose an isogeny
	\[
	f:A\longrightarrow E_1\times\cdots\times E_n,
	\]
	where $E_i$ are elliptic curves. Let $U_i$ be a positive-degree line bundle on $E_i$. Let $\text{pr}_i: E_1\times\cdots\times E_n \longrightarrow E_i$ be the $i$-th projection. Set $N_i:=f^*\operatorname{pr}_i^*U_i$.
	Choose $m \gg 0$ such that
	\[
	H=\bigg(\bigotimes_{i=1}^nN_i\bigg)^{\otimes m}
	\]
	is a very ample line bundle on $A$. We denote $L_0:=\bigotimes_{i=1}^nN_i^{\otimes mi}$.
	Then
	\[
	(L_0 \otimes H^{-j})=\bigg(\bigotimes_{i=1}^nN_i^{m(i-j)}\bigg).
	\]
	The $j$-th tensor factor is trivial, so $L_0\otimes H^{-j}$ is pulled back from the product of the remaining $n-1$ elliptic curves. 
	Thus, for each $1\leq j\leq n$, the top self-intersection $(L_0-jH)^n=0$.
	
	Define $M_{j,0}:=L_0 \otimes H^{-j}$. Then $c_1(M_{j,0})^n=(L_0-jH)^n=0$. Take 
	\[
	K_j:=\bigl(\ker\phi_{M_{j,0}}\bigr)^0_{\mathrm{red}}.
	\]
	Because $H$ is very ample and $c_1(M_{j,0})^n=0$ for every $1\le j\le n$, the numerical implication in Theorem \ref{sec5:thm1} shows that each $K_j$ is an elliptic curve and that their addition map $s:\prod_{j=1}^nK_j\to A$ is an isogeny.
	
	Denote $\eta_{j,0}:=M_{j,0}|_{K_j}$. Since $K_j \subseteq\ker \phi_{M_{j, 0}}$, for every $x \in K_j(k)$,
	\[
	t_x^*\eta_{j, 0} \cong (t_x^*M_{j, 0})|_{K_j} \cong M_{j, 0}|_{K_j}=\eta_{j, 0},
	\]
	where the translations are taken on $K_j$ and $A$, respectively. Thus $\phi_{\eta_{j, 0}}=0$. Since $K_j$ is an elliptic curve, this gives $\deg \eta_{j, 0}=0$; see \cite[Definition of $\text{Pic}^0A$, pp. 35-36]{milneAV}. Thus, $\eta_{j,0} \in \operatorname{Pic}^0(K_j)$.
	
	The dual isogeny
	\[
	s^\vee:\widehat A \longrightarrow \prod_{j=1}^n\widehat K_j, \qquad P\longmapsto(P|_{K_j})_{j=1}^n,
	\]
	is surjective; see \cite[Theorem 9.1, Proposition 7.1]{milneAV}. Choose $Q_j\in\operatorname{Pic}^0(K_j)$ with
	$Q_j^{\otimes p}\not\cong\mathcal O_{K_j}$,
	which is possible because the group of $p$-torsion elements $\widehat K_j[p]$ is finite. Hence there exists $P\in\widehat A(k)$ such that
	\[
	P|_{K_j}\cong\eta_{j,0}^{-1}\otimes Q_j \qquad(1\leq j\leq n).
	\]
	Set $L:=L_0\otimes P$ and $M_j:=L\otimes H^{-j}$. Since $P\in\operatorname{Pic}^0(A)$, we have $\phi_P=0$; see \cite[Definition of $\text{Pic}^0A$, p. 35-36]{milneAV}. This gives
	\begin{equation} \label{laste61}
		\phi_{M_j}=\phi_{M_{j,0}}+\phi_P=\phi_{M_{j,0}}, \qquad M_j|_{K_j}\cong M_{j, 0} \vert_{K_j} \otimes P\vert_{K_j} \cong Q_j.
	\end{equation}
	Thus, we get $K_j=\bigl(\ker\phi_{M_{j}}\bigr)^0_{\mathrm{red}}$. Since $(M_j \vert_{K_j})^{\otimes p} \cong Q_j^{\otimes p} \ncong \cO_{K_j}$, Lemma \ref{lasrkeylem1} shows that
	\[
	H^q(A,M_j) = H^q(A,M_j^{\otimes p})=0,
	\]
	for every $q$ and for all $1 \le j \le n$. Thus, $B^1_A \otimes L$ is Ulrich \wrt $H$ by Corollary \ref{lcor1}.
\end{proof}

	\begin{cor} \label{sec5:cor2}
		Let $A$ be an abelian variety of dimension $n\geq2$. If $\rho(A)<n$, then there do not exist a line bundle $L$ and a very ample divisor $H$ on $A$ such that $B_A^1\otimes L$ is Ulrich with respect to $H$.
		\end{cor}		
		
		\begin{proof}
Suppose, for a contradiction, that such $L$ and $H$ exist. By Theorem \ref{sec5:thm1}, there exists an isogeny $s:B:=\prod_{j=1}^nK_j \longrightarrow A$ such that each $K_i$ is an elliptic curve. Since the Picard number is invariant under isogenies, we have $\rho(A)=\rho(B)$; see \cite[Section 2.2]{laface2018picard}. Since $\rho(B) \geq \sum_{i=1}^n\rho(K_i)$ by \cite[Proposition 1.7, p. 753]{skorobogatov2014brauer}, we get $\rho(A) \ge n$, giving a contradiction.
		\end{proof}

			\subsection*{Acknowledgements} 
			The author is grateful to A. J. Parameswaran for several very helpful discussions during the initial stages of this work, while visiting the Kerala School of Mathematics (KSoM), Kerala. The author thanks KSoM for its hospitality during this visit. The author also thanks him for his comments on this manuscript. The author was supported by a Postdoctoral Fellowship at TIFR, Mumbai, where this work was initiated, and was subsequently supported by a Visiting Scientist position at ISI Bangalore. The author is currently supported by a Postdoctoral Fellowship at IISER Pune.

	\bibliographystyle{plainurl}
	\bibliography{ulrich.bib}

	\end{document}